\documentclass[11pt]{article}
\usepackage[margin=1in]{geometry}  % set the margins to 1in on all sides
\usepackage{amsmath}
\usepackage{cases}
\usepackage{stmaryrd}
\usepackage{latexsym,amsfonts,amssymb,amsmath,amsthm}
\usepackage{mathrsfs}
\usepackage{graphicx}
\usepackage{verbatim}

\newtheorem{thm}{Theorem}[section]

\newtheorem{lem}[thm]{Lemma}

\theoremstyle{definition}
\newtheorem{defn}[thm]{Definition}
\theoremstyle{remark}
\newtheorem{rem}[thm]{Remark}

\numberwithin{equation}{section}

\begin{document}

\nocite{*}

\title{Existence, Uniqueness and Anisotropic-Decay-Caused Lifshitz Tails\\ of the Integrated Density of Surface States for \\Random Surface Models}

\author{Zhongwei Shen\footnote{Email: zzs0004@auburn.edu}\\Department of Mathematics and Statistics\\Auburn University\\Auburn, AL 36849\\USA}

\date{}

\maketitle

\begin{abstract}
    The current paper is devoted to the study of existence, uniqueness and Lifshitz tails of the integrated density of surface states (IDSS) for Schr\"{o}dinger operators with alloy type random surface potentials. We prove the existence and uniqueness of the IDSS for negative energies, which is defined as the thermodynamic limit of the normalized eigenvalue counting functions of localized operators on strips with sections being special cuboids. Under the additional assumption that the single-site impurity potential decays anisotropically, we also prove that the IDSS for negative energies exhibits Lifshitz tails near the bottom of the almost sure spectrum in the following three regimes: the quantum regime, the quantum-classical/classical-quantum regime and the classical regime. We point out that the quantum-classical/classical-quantum regime is new for random surface models.
\end{abstract}

\tableofcontents

\section{Introduction}

This paper is concerned with the existence, uniqueness and Lifshitz tails (or Lifshitz singularities, or Lifshitz behavior) of the integrated density of surface states (IDSS) for the following random Schr\"{o}dinger operator
\begin{equation}\label{random-surface-model}
H_{\omega}=-\Delta+V_{0}+V_{\omega}\quad\text{on}\quad\mathbb{R}^{d+n},
\end{equation}
where $-\Delta$ is the negative Laplacian, $V_{0}$ is the bulk potential used to model a perfect crystal and $V_{\omega}$ is the random surface potential of alloy type concentrated near the $d$-dimensional surface $\mathbb{R}^{d}\times\{0\}\subset\mathbb{R}^{d}\times\mathbb{R}^{n}$, that is, $V_{\omega}$ has the form
\begin{equation}\label{random-surface-potential}
V_{\omega}(x,y)=\sum_{i\in\mathbb{Z}^{d}}\omega_{i}f(x-i,y),\quad x\in\mathbb{R}^{d}\,\,\text{and}\,\,y\in\mathbb{R}^{n}
\end{equation}
where $\{\omega_{i}\}_{i\in\mathbb{Z}^{d}}$ are independent and identically distributed (i.i.d) random variables on some probability space $\Omega$ and $f$ is the single-site impurity potential. See Section \ref{assumptions-and-results} for exact assumptions.

Operator \eqref{random-surface-model} is used to model non-interacting electrons in a crystal with additional random impurities. A vast amount of literature has been carried out toward the spectral structure on the random Schr\"{o}dinger operator \eqref{random-surface-model} as well as its discrete analog. See \cite{BS03,BKS05,BSS05,CS00,HK00,JL00,JL01,JM99,JMP98} and references therein. However, there is very few work on Lifshitz tails for the random Schr\"{o}dinger operator \eqref{random-surface-model} and its discrete version. See \cite{KK06} for the discrete model and \cite{KW06} for the continuum model. As a motivation for the current paper, we roughly describe results obtained in \cite{KW06} by Kirsch and Warzel. In \cite{KW06}, they studied the existence, uniquness and Lifshitz tails of the IDSS for the model \eqref{random-surface-model} (with one more ergodic term). For the existence and uniqueness of the IDSS, they proved the existence of the limit $N^{X}(E):=\lim_{L\rightarrow\infty}\frac{1}{|\Lambda_{L}|}N(H_{\omega,S_{L}}^{X},E)$ for $E<0$ (see Section \ref{assumptions-and-results} for the definition of $N(H_{\omega,S_{L}}^{X},E)$) and the uniqueness $N^{D}(E)=N^{N}(E)$ for $E<0$, where $\Lambda_{L}$ is the open cube in $\mathbb{R}^{d}$ centered at $0$ with side length $L$ and $S_{L}=\Lambda_{L}\times\mathbb{R}^{n}$. The IDSS $N(E)$ for negative energies $E<0$ is defined to be the common values. We remark that the fact $\Lambda_{L}$ is a cube plays an important role, since their proof relies heavily on the eigenvalues of the Neumann Laplacian on cubes. For Lifshitz tails, under the assumption that the single-site impurity potential $f:\mathbb{R}^{d+n}\rightarrow[0,\infty)$ decays isotropically in the $x$-direction and is uniformly bounded in the $y$-direction, that is, $f$ satisfies $f_{u}|x|^{-\alpha}\chi_{G}(y)\leq f(x,y)\leq f_{0}|x|^{-\alpha}$ for $|x|$ large and any $y\in\mathbb{R}^{n}$, they showed that the IDSS for negative energies exhibits Lifshitz tails near the bottom of the spectrum with
\begin{equation*}
\lim_{E\downarrow E_{0}}\frac{\ln|\ln
N(E)|}{\ln(E-E_{0})}=-\max\bigg\{\frac{d}{2},\frac{d}{\alpha-d}\bigg\}.
\end{equation*}
Cases $\alpha\geq d+2$ and $d<\alpha<d+2$ correspond to the quantum regime and the classical regime, respectively. We remark that Lifshitz tails in the classical regime are also called Pastur tails.

Another motivation for this paper is the work of Kirsch and Warzel \cite{KW05}. In \cite{KW05}, they studied the Lifshitz tails for a class of general random operators, which cover the operator \eqref{random-surface-model} in the case $n=0$. The main assumption on the single-site impurity potential $f:\mathbb{R}^{d}\rightarrow[0,\infty)$ is the anisotropic decay, that is, $f(x_{1},x_{2})\sim(|x_{1}|^{\alpha_{1}}+|x_{2}|^{\alpha_{2}})^{-1}$ as $|(x_{1},x_{2})|\rightarrow\infty$, where $(x_{1},x_{2})\in\mathbb{R}^{d_{1}}\times\mathbb{R}^{d_{2}}$ with $d_{1}+d_{2}=d$. They proved that the integrated density of states (IDS) exhibits Lifshitz tails near the bottom of the spectrum with
\begin{equation*}
\lim_{E\downarrow E_{0}}\frac{\ln|\ln
N(E)|}{\ln(E-E_{0})}=-\max\bigg\{\frac{d_{1}}{2},\frac{\gamma_{1}}{1-\gamma}\bigg\}-\max\bigg\{\frac{d_{2}}{2},\frac{\gamma_{2}}{1-\gamma}\bigg\},
\end{equation*}
where $\gamma_{k}=\frac{d_{k}}{\alpha_{k}}$, $k=1,2$ and $\gamma=\gamma_{1}+\gamma_{2}$. Cases $\frac{d_{k}}{2}>\frac{\gamma_{k}}{1-\gamma}$, $k=1,2$ and
$\frac{d_{k}}{2}\leq\frac{\gamma_{k}}{1-\gamma}$, $k=1,2$ correspond to the quantum regime and the classical regime, respectively. The other two cases: (i) $\frac{d_{1}}{2}>\frac{\gamma_{1}}{1-\gamma}$ and $\frac{d_{2}}{2}\leq\frac{\gamma_{2}}{1-\gamma}$; (ii) $\frac{d_{1}}{2}\leq\frac{\gamma_{1}}{1-\gamma}$ and $\frac{d_{2}}{2}>\frac{\gamma_{2}}{1-\gamma}$, correspond to the quantum-classical/classical-quantum regime (since the quantum-classical regime and the classical-quantum regime are essentially symmetric, we here use ``the quantum-classical/classical-quantum regime" to denote both of them), which is unknown before them. In conclusion, they recovered the classical results (the quantum regime and the classical regime) with $f$ being isotropic decay and found a new regime (the quantum-classical/classical-quantum regime).

Inspired by the work of Kirsch and Warzel \cite{KW05,KW06}, we study the existence, uniqueness and Lifshitz tails caused by anisotropic decay of the IDSS for the random Schr\"{o}dinger operator \eqref{random-surface-model}. Main results of the paper can be roughly summarized as follows.
\begin{itemize}
    \item[\rm(i)] We prove the existence and uniqueness of the IDSS for negative energies, which is defined as the thermodynamic limit of the eigenvalue counting functions of localized operators on strips of the form $\Lambda\times\mathbb{R}^{n}$, where $\Lambda\subset\mathbb{R}^{d}$ are special open cuboids. Moreover, we justify that the IDSS for negative energies obtained in the current paper coincides with the one obtained in \cite{KW06} by Kirsch and Warzel. It is worthwhile to point out that the uniqueness of the IDSS is unknown for $\Lambda$ being general domains for the reason that the proof in \cite{KW06} or in Subsection \ref{subsec-proof} depends heavily on the eigenvalues of the Neumann Laplacian on $L^{2}(\Lambda)$.

    \item[\rm(ii)] Under the anisotropic decay of the single-site impurity potential $f:\mathbb{R}^{d+n}\rightarrow[0,\infty)$, that is, $f$ satisfies $f_{u}(|x_{1}|^{\alpha_{1}}+|x_{2}|^{\alpha_{2}})^{-1}\chi_{G}(y)\leq f(x,y)\leq f_{0}(|x_{1}|^{\alpha_{1}}+|x_{2}|^{\alpha_{2}})^{-1}$ for $|(x_{1},x_{2})|$ large and any $y\in\mathbb{R}^{n}$, we prove that the IDSS for negative energies exhibits Lifshitz tails near the bottom of the spectrum in all three regimes: the quantum regime, the quantum-classical/classical-quantum regime and the classical regime. Hence, we recover the results obtained in \cite{KW06} by Kirsch and Warzel  with $f$ decaying isotropically and find the quantum-classical/classical-quantum regime, which is new for random surface models.
\end{itemize}

We remark that for Lifshitz tails in both the quantum regime and the classical regime, we only need to study the IDSS by means of localized operators on strips $\Lambda\times\mathbb{R}^{n}$ with $\Lambda$ being open cubes in $\mathbb{R}^{d}$. It is the Lifshitz tails in the quantum-classical/classical-quantum regime forcing us to study the IDSS using localized operators on strips $\Lambda\times\mathbb{R}^{n}$ with $\Lambda$ being open cuboids in $\mathbb{R}^{d}$. Besides the above two main results $\rm(i)$ and $\rm(ii)$, we also prove the estimate of the spectral gap between
the lowest two eigenvalues of the localized partially periodic operator on strips as it was proven in \cite{KW06}, which plays a crucial role in the study of Lifshitz tails.

It should be pointed out that besides the study of Lifshitz tails for random surface models, Lifshitz tails for other random operators have been widely studied and proven to exist near the bottom of the spectrum since the first proof, given by Donsker and Varadhan \cite{DV75}, of Lifshitz's prediction \cite{Li63, Li65}. See \cite{KW05,KM83,KS86,Me87,PF92}, etc. for random alloy-type models, \cite{BHKL95,Er98,Er01,HKW03,HLW99,KLW,Kl10,KR06,LW04,Wa01}, etc. for random Landau Hamiltonians and \cite{KM06,MS07}, etc. for percolation models. Lifshitz tails are also shown to exist near band
edges of the spectrum as predicted by Lifshitz. This phenomenon is now referred to as internal Lifshitz tails. See \cite{Gh07,Kl99,Kl01,Kl02,KW02,Me86,Me93,Si87} and references therein. The survey paper \cite{KM07} provides a quite complete summary of above results. There are also results on Lifshitz tails for random magnetic fields. See \cite{Gh07,Na00-1,Na00-2,Ue00,Ue02} and references therein. Other types of random operators such as random wave operators, random block operators, hierarchical Anderson model, etc. were also shown to exhibit Lifshitz tails. See
\cite{KMM11,KV10,MK12,Na03,Na07} and references therein. Recently, Lifshitz tails were shown to exist in non-monotonous alloy type
random Schr\"{o}dinger operators. See \cite{Gh08,KN09,KN10} and references therein.

The rest of the paper is organized as follows. In section \ref{assumptions-and-results}, we give standard assumptions on the random surface model \eqref{random-surface-model} and state main results of this paper. In Section \ref{IDSS}, we prove the existence and uniqueness of the IDSS for negative energies. Section \ref{sec-spectral-gap} is devoted to the preparation for the proof of Lifshitz tails. In which, we prove the crucial spectral gap estimates and obtain a sandwiching bound for the IDSS for negative energies. In Section \ref{proof-of-Lifshitz-tails}, we prove the existence of Lifshitz tails near the bottom of the spectrum for the random surface model.

\section{Notations, Assumptions and Main Results}\label{assumptions-and-results}

In this section, we give basic assumptions on the random surface model \eqref{random-surface-model}, i.e., assumptions on the bulk potential $V_{0}$ and the random surface potential $V_{\omega}$, and state main results regarding the existence, uniqueness and Lifshitz tails of the IDSS.

We first make some conventions for the discrete spectrum of a self-adjoint operator. For any self-adjoint operator $H$, its spectrum is denoted by $\sigma(H)$. If $H$ has discrete spectrum below its essential spectrum, the discrete spectrum below the essential spectrum are denoted by $E_{0}(H)\leq E_{1}(H)\leq E_{2}(H)\leq\cdots$ according to multiplicity. Moreover, if the discrete spectrum below the essential spectrum consists of the points $E_{0}(H),E_{1}(H),\dots,E_{n-1}(H)$, we denote by $E_{n}(H)$ the bottom of the essential spectrum.

For self-adjoint operators restricted to subdomains with self-adjoint boundary conditions, we will frequently use the following notations. Suppose $H$ is a self-adjoint operator on $L^{2}(\mathbb{R}^{d+n})$. Let $\Lambda\subset\mathbb{R}^{d}$ be an open set (in particular, $\Lambda$ is a cuboid in $\mathbb{R}^{d}$) and let $S=\Lambda\times\mathbb{R}^{n}$ be the strip. We denote by $H_{S}^{X}$ the operator $H$ restricted to $L^{2}(S)$ with $X$
boundary condition on $\partial S$, where $X=D$ or $X=N$ refers to Dirichlet or Neumann boundary condition. Also, if we let $\widetilde{\Lambda}\subset\mathbb{R}^{n}$ be an open set and set $\widetilde{S}=\Lambda\times\widetilde{\Lambda}$, then the notation $H^{X,Y}_{\widetilde{S}}$ is used to stand for the operator $H$ restricted to $L^{2}(\widetilde{S})$ with $X$ boundary condition on $\partial\Lambda\times\widetilde{\Lambda}$ and $Y$ boundary condition on $\Lambda\times\partial\widetilde{\Lambda}$, where $X,Y=D$ or $N$ refer to either Dirichlet or Neumann boundary condition.

We next make assumptions to ensure the self-adjointness of operators $H_{\omega}$, $\omega\in\Omega$. Suppose

\begin{itemize}
    \item[\rm(H1)] $V_{0}\in\mathcal{K}(\mathbb{R}^{d+n})\cap L^{2}_{\rm loc}(\mathbb{R}^{d+n})$ is real-valued and $\mathbb{Z}^{d}$-periodic, that is,
        \begin{equation*}
        V_{0}(x+i,y)=V_{0}(x,y)\quad\text{for all}\,\,x\in\mathbb{R}^{d},y\in\mathbb{R}^{n}\,\,\text{and}\,\,i\in\mathbb{Z}^{d},
        \end{equation*}
        where $\mathcal{K}(\mathbb{R}^{d+n})$ is the Kato class (see \cite{Si82}) and $L^{2}_{\rm loc}(\mathbb{R}^{d+n})$ is the space of locally square integrable complex-valued functions on $\mathbb{R}^{d+n}$. The above assumption guarantees that $-\Delta+V_{0}$ is self-adjoint and is called the bulk operator. By shifting the energy, we assume without loss of generality that $\inf\sigma(-\Delta+V_{0})=0$.
\end{itemize}

and

\begin{itemize}
    \item[\rm(H2)] $V_{\omega}$ is the random alloy-type surface potential having the form
        \begin{equation*}
            V_{\omega}(x,y)=\sum_{i\in\mathbb{Z}^{d}}\omega_{i}f(x-i,y),\quad x\in\mathbb{R}^{d}\,\,\text{and}\,\,y\in\mathbb{R}^{n}
        \end{equation*}
where
\begin{itemize}
    \item[(i)] $\{\omega_{i}\}_{i\in\mathbb{Z}^{d}}$ are independent and identically distributed (i.i.d) random variables on some probability space $(\Omega,\mathcal{B},\mathbb{P})$ with common distribution $\mathbb{P}_{0}$. We assume that the support of $\mathbb{P}_{0}$, denoted by $\text{supp}\mathbb{P}_{0}$, is compact, contains at least two points and is contained in $(-\infty,0)$. By the canonical realization of stochastic processes, we may take $\Omega=(\text{supp}\mathbb{P}_{0})^{\mathbb{Z}^{d}}$, and thus, $\mathbb{P}$ is the product measure $\otimes_{i\in\mathbb{Z}^{d}}\mathbb{P}_{0}$. We denote by $\mathbb{E}$ the expectation corresponding to $\mathbb{P}$.

    \item[(ii)] The single-site impurity potential $f:\mathbb{R}^{d+n}\rightarrow[0,\infty)$ is positive on a nonempty open set in $\mathbb{R}^{d+n}$. More precisely, there exist a constant $f_{u}>0$ and two Borel sets $F\subset\mathbb{R}^{d}$ and $G\subset\mathbb{R}^{n}$ such that $f(x,y)\geq f_{u}\chi_{F}(x)\chi_{G}(y)$. By shifting $f$ along $\mathbb{Z}^{d}$-direction and making $F$ smaller, we may assume that $F\subset\Lambda_{1}$, where $\Lambda_{1}$ is the unit open cube in $\mathbb{R}^{d}$ centered at $0\in\mathbb{R}^{d}$.

    \item[(iii)] We also assume $f\in\ell^{1}(L^{p}(\mathbb{R}^{d+n}))$, the Birman-Solomyak space, with $p\geq2$ and $p>d+n$.
\end{itemize}
\end{itemize}

For the self-adjointness of $H_{\omega}$, $\omega\in\Omega$, the assumption $f\in\ell^{1}(L^{p}(\mathbb{R}^{d+n}))$ with $p\geq2$ and $p>d+n$ in $\rm(H2)(iii)$ is a little stronger, but we need this stronger assumption for imposing boundary conditions (see \cite[Assumption 2.7]{KW05} and \cite[Theorem C.2.4]{Si82}).

Let $\omega_{\min}=\inf\text{supp}\mathbb{P}_{0}$, we define $V:\mathbb{R}^{d+n}\rightarrow(-\infty,0]$ by
\begin{equation*}
V(x,y)=\omega_{\min}\sum_{i\in\mathbb{Z}^{d}}f(x-i,y).
\end{equation*}
and assume
\begin{itemize}
    \item[\rm(H3)]$\inf_{x\in\mathbb{R}^{d}}V(x,y)\rightarrow0$ as $|y|\rightarrow\infty$.
\end{itemize}

Assumption $\rm(H3)$ is used to guarantee the applicability of Weyl's theorem (see e.g. \cite[Theorem XIII.14]{RS78}) on the stability of essential spectrum. Moreover, assumption $\rm(H2)$ and $\rm(H3)$ ensure that $V\in L^{p}_{\rm unif,loc}(\mathbb{R}^{d+n})\subset\mathcal{K}(\mathbb{R}^{d+n})$
with $p\geq2$ and $p>d+n$.

Under above assumptions, we are able to prove the following fundamental results.

\begin{lem}
Suppose $\rm(H1)$, $\rm(H2)$ and $\rm(H3)$. There hold the following statements.
\begin{itemize}
    \item[\rm(i)] $H_{\omega}$, $\omega\in\Omega$ is almost surely essentially self-adjoint on $\mathcal{C}_{0}^{\infty}(\mathbb{R}^{d+n})$;

    \item[\rm(ii)] $H_{\omega}$, $\omega\in\Omega$ is $\mathbb{Z}^{d}$-ergodic. Hence, there's $\Sigma\subset\mathbb{R}$ such that $\sigma(H_{\omega})=\Sigma$ a.e. $\omega\in\Omega$;

    \item[\rm(iii)] Let $E_{0}=\inf\sigma(H_{\rm per})$, where
        \begin{equation}\label{periodic-operator}
        H_{\rm per}=-\Delta+V_{0}+V
        \end{equation}
Then $\inf\Sigma=E_{0}$, that is, $\inf\sigma(H_{\omega})=\inf\sigma(H_{\rm per})$ a.e. $\omega\in\Omega$.
\end{itemize}
\end{lem}
\begin{proof}
See \cite[Proposition 1.1]{KW06} for $\rm(i)$ and $\rm(ii)$, and \cite[Proposition 1.2]{KW06} for $\rm(iii)$.
\end{proof}

To study the IDSS for negative energies, or below the bulk spectrum $\sigma(-\Delta+V_{0})$ (by assumption $\rm(H1)$, $\inf\sigma(-\Delta+V_{0})=0$), we assume
\begin{itemize}
    \item[\rm(H4)] The ground state energy of $H_{per}$, or the bottom of the almost sure spectrum of $H_{\omega}$, is negative, that is, $E_{0}<0$.
\end{itemize}

Assumption \rm{(H4)} is readily satisfied if $\omega_{\min}$, hence $V$, is negative enough because of Hardy's inequality (see \cite{OK90} for example).

Finally, we state our main results. Recall that $S=\Lambda\times\mathbb{R}^{n}$ with $\Lambda$ being an open bounded set in $\mathbb{R}^{d}$ and $H_{\omega,S}^{X}$ denotes the operator $H_{\omega}$ restricted to $L^{2}(S)$ with $X$ boundary condition on $\partial S$. For $E<0$, we define the eigenvalue counting function
\begin{equation*}
N\big(H_{\omega,S}^{X},E\big):=\#\Big\{n\in\mathbb{N}_{0}\Big|E_{n}(H_{\omega,S}^{X})\leq E\Big\},
\end{equation*}
where $\mathbb{N}_{0}=\mathbb{N}\cup\{0\}$ and $\#\{\cdot\}$ is the cardinal number of the set $\{\cdot\}$. We remark that $N\big(H_{\omega,S}^{X},E\big)$ is almost surely finite for any $E<0$ due to the fact that the essential spectrum of $H_{\omega,S}^{X}$ is contained in $[0,\infty)$ by $\rm(H1)$, $\rm(H3)$ and Weyl's essential spectrum theorem (see e.g. \cite[Theorem XIII.14]{RS78}). For the set $\Lambda$ in $\mathbb{R}^{d}$, we consider the following three kinds:
\begin{itemize}
    \item[\rm (i)] cubes: $\Lambda_{L}=\big(-\frac{L}{2},\frac{L}{2}\big)^{d}$, $L\geq1$,

    \item[\rm (ii)] cuboids: $\Lambda_{K}^{1}(L)=\big(-\frac{K+L-1}{2},\frac{K+L-1}{2}\big)^{d_{1}}\times\big(-\frac{K}{2},\frac{K}{2}\big)^{d_{2}}$, $L,K\geq1$,

    \item[\rm (iii)] cuboids: $\Lambda_{K}^{2}(L)=\big(-\frac{K}{2},\frac{K}{2}\big)^{d_{1}}\times\big(-\frac{K+L-1}{2},\frac{K+L-1}{2}\big)^{d_{2}}$, $L,K\geq1$.
\end{itemize}
The corresponding strip $S$ are denoted by $S_{L}$, $S_{K}^{1}(L)$ and $S_{K}^{2}(L)$, respectively. With the help of above notations, we are able to state our first main result regarding the existence and uniqueness of the IDSS.

\begin{thm}\label{theorem-IDSS}
Suppose $\rm(H1)$, $\rm(H2)$, $\rm(H3)$ and $\rm(H4)$.
\begin{itemize}
    \item[\rm (i)] For $E<0$, the limit \begin{equation*}N^{X}(E):=\lim_{L\rightarrow\infty}\frac{N\big(H_{\omega,S_{L}}^{X},E\big)}{|\Lambda_{L}|}\end{equation*}exists and almost surely non random. Moreover, $N^{D}(E)=N^{N}(E)$ for all but possible countably many $E<0$.

    \item[\rm (ii)] Let $k\in\{1,2\}$ and $L\geq1$. For $E<0$, the limit \begin{equation*}N_{k,L}^{X}(E):=\lim_{K\rightarrow\infty}\frac{N\big(H_{\omega,S_{K}^{k}(L)}^{X},E\big)}{|\Lambda_{K}^{k}(L)|}\end{equation*}exists and almost surely non random. Moreover, $N_{k,L}^{D}(E)=N_{k,L}^{N}(E)$ for all $L\geq1$ and all but possible countably many $E<0$.

    \item[\rm (iii)] If we denote the common values obtained in $\rm(i)$ and $\rm (ii)$ by $N$ and $N_{k,L}$, $k\in\{1,2\}$, $L\geq1$, respectively, then we have $N(E)=N_{k,L}(E)$ for $k\in\{1,2\}$, all $L\geq1$ and all but possible countably many $E<0$.
\end{itemize}
\end{thm}

The proof of the above theorem is given in Subsection \ref{subsec-proof}. Given Theorem \ref{theorem-IDSS}, we make the following definition.

\begin{defn}
$N(E)$ is well-defined for all but possible countably many $E<0$ and it is called the integrated density of surface states for negative energies for $H_{\omega}$, $\omega\in\Omega$.
\end{defn}

We remark that there are other ways to define the IDSS. See \cite{EKSS88,EKSS90,KS00,KS01} and references therein. In these literature, the IDSS is defined for all energies and, for negative energies, coincides with the definition above. We refer to \cite{KW06} for more discussions.

To state another main result, we make additional assumptions on both the single-site impurity potential $f$ and the common probability measure $\mathbb{P}_{0}$.

\begin{itemize}
\item[\rm(H5)] Let $d_{1},d_{2}\in\mathbb{N}$ be such that $d=d_{1}+d_{2}$. There exist $f_{0}>0$, $f_{u}>0$, $\alpha_{1}>d_{1}$, $\alpha_{2}>d_{2}$ and a nonempty Borel set $G\subset\mathbb{R}^{n}$ with nonzero finite Lebesgue measure such that
\begin{equation*}
\frac{f_{u}}{(1+|x_{1}|)^{\alpha_{1}}+(1+|x_{2}|)^{\alpha_{2}}}\chi_{G}(y)\leq f(x,y)\leq\frac{f_{0}}{(1+|x_{1}|)^{\alpha_{1}}+(1+|x_{2}|)^{\alpha_{2}}}
\end{equation*}
for all $x\in\mathbb{R}^{d}$ and $y\in\mathbb{R}^{n}$, where $x=(x_{1},x_{2})$ with $x_{k}\in\mathbb{R}^{d_{k}}$, $k=1,2$.
\end{itemize}

\begin{itemize}
\item[\rm(H6)] There are constants $C>0$, $N>0$ and $\epsilon_{0}>0$ such that
\begin{equation*}
\mathbb{P}_{0}\Big\{[\omega_{\min},\omega_{\min}+\epsilon)\Big\}\geq C\epsilon^{N}
\end{equation*}
for all $\epsilon\in(0,\epsilon_{0}]$.
\end{itemize}

Assumption $\rm(H5)$ is referred to as the anisotropic decay of $f$, i.e., anisotropic decay in the $x$-direction and uniform boundedness in the $y$-dirction. This assumption determines the asymptotic behavior of $N(E)$ near $E_{0}$. $\rm(H6)$ is a technical assumption, which is used to obtain a lower bound in the proof of Lifshitz tails.

We now state the main result regarding the asymptotic behavior of $N(E)$, $E<0$ near the bottom of the spectrum, i.e., $E_{0}$.

\begin{thm}\label{main-theorem}
Suppose $\rm(H1)$, $\rm(H2)$, $\rm(H3)$, $\rm(H4)$, $\rm(H5)$ and $\rm(H6)$. Let $\gamma_{k}=\frac{d_{k}}{\alpha_{k}}$, $k=1,2$ and $\gamma=\gamma_{1}+\gamma_{2}$. Consider the following three regimes:
\begin{itemize}
\item[\rm(i)] quantum regime: $\frac{d_{1}}{2}>\frac{\gamma_{1}}{1-\gamma}$ and $\frac{d_{2}}{2}>\frac{\gamma_{2}}{1-\gamma}$;

\item[\rm(ii)] quantum-classical/classical-quantum regime:
\begin{equation*}
\frac{d_{1}}{2}>\frac{\gamma_{1}}{1-\gamma}\,\,\text{and}\,\,\frac{d_{2}}{2}\leq\frac{\gamma_{2}}{1-\gamma},\,\,\text{or}\,\,\frac{d_{1}}{2}\leq\frac{\gamma_{1}}{1-\gamma}\,\,\text{and}\,\,
\frac{d_{2}}{2}>\frac{\gamma_{2}}{1-\gamma};
\end{equation*}

\item[\rm(iii)] classical regime: $\frac{d_{1}}{2}\leq\frac{\gamma_{1}}{1-\gamma}$ and $\frac{d_{2}}{2}\leq\frac{\gamma_{2}}{1-\gamma}$.
\end{itemize}
Then, the integrated density of surface states $N(E)$ for negative energies $E<0$ exhibits Lifshitz tails near $E_{0}$ in all three regimes with
\begin{equation}\label{lifshitz-tails}
\lim_{E\downarrow E_{0}}\frac{\ln|\ln N(E)|}{\ln(E-E_{0})}=-\max\bigg\{\frac{d_{1}}{2},\frac{\gamma_{1}}{1-\gamma}\bigg\}-\max\bigg\{\frac{d_{2}}{2},\frac{\gamma_{2}}{1-\gamma}\bigg\}.
\end{equation}
\end{thm}

The proof of the above theorem is given in Section \ref{proof-of-Lifshitz-tails}. We end this section by making a remark about Theorem \ref{main-theorem}.

\begin{rem}\label{remark-main-theorem}
Results similar to Theorem \ref{theorem-IDSS} can be proven with $f$ anisotropically decaying in a more general way. That is, if $f$ satisfies
\begin{equation*}
\frac{f_{u}}{\sum_{k=1}^{m}(1+|x_{k}|)^{\alpha_{k}}}\chi_{G}(y)\leq f(x,y)\leq\frac{f_{0}}{\sum_{k=1}^{m}(1+|x_{k}|)^{\alpha_{k}}}
\end{equation*}
for all $x\in\mathbb{R}^{d}$ and $y\in\mathbb{R}^{n}$, where $x=(x_{1},\dots,x_{m})$ with $x_{k}\in\mathbb{R}^{d_{k}}$, $k=1,\dots,m$ and $d=\sum_{k=1}^{m}d_{k}$, then
\begin{equation*}
\lim_{E\downarrow E_{0}}\frac{\ln|\ln N(E)|}{\ln(E-E_{0})}=-\sum_{k=1}^{m}\max\bigg\{\frac{d_{k}}{2},\frac{\gamma_{k}}{1-\gamma}\bigg\},
\end{equation*}
where $\gamma_{k}=\frac{d_{k}}{\alpha_{k}}$, $k=1,2,\dots,m$ and $\gamma=\sum_{k=1}^{m}\gamma_{k}$.
\end{rem}

%%%%%%%%%%%%%%%%%%%%%%%%%%%%%%%%%%%%%%%%%%%%%%%%%%%%%%%%%%%%%%%%%%%%%%%%%%%%%%%%%%%%%%%%%%%%%%%%%%%%%%%%%%%%%%%%%%%%

\section{The Integrated Density of Surface States for Negative Energies}\label{IDSS}

This section is devoted to the study of existence and uniqueness of the IDSS for $H_{\omega}$, $\omega\in\Omega$, that is, we prove Theorem \ref{theorem-IDSS}. Throughout this section, assumptions $\rm(H1)$, $\rm(H2)$, $\rm(H3)$ and $\rm(H4)$ are assumed to be satisfied.

\subsection{Proof of Theorem \ref{theorem-IDSS}}\label{subsec-proof}

Note that Theorem \ref{theorem-IDSS} $\rm (i)$ is a special case of \cite[Theorem 1.3]{KW06}. The proof of $\rm (ii)$ and $\rm (iii)$ in Theorem \ref{theorem-IDSS} are broken into several parts.

\begin{thm}\label{theorem-IDSS-cuboid}
Let $k\in\{1,2\}$ and $L\geq1$. Then,
\begin{itemize}

\item[\rm(i)] for $E<0$, the limit
\begin{equation*}\label{IDSS-cuboids}
N_{k,L}^{X}(E):=\lim_{K\rightarrow\infty}\frac{N\big(H_{\omega,S_{K}^{k}(L)}^{X},E\big)}{|\Lambda_{K}^{k}(L)|}
\end{equation*}
exists and almost surely non random;

\item[\rm(ii)] $N_{k,L}^{D}(E)=N_{k,L}^{N}(E)$ for all but possible countably many $E<0$.
\end{itemize}
\end{thm}
\begin{proof}
We focus on the case $k=1$, since the results in the case $k=2$ can be proven in a similar manner. $\rm(i)$ is a simple consequence of the Akcoglu-Krengel ergodic theorem (see e.g. \cite{AK81}, \cite{KM82}). To prove the $\rm(ii)$, we first prove some lemmas.
\end{proof}

Let
\begin{equation*}
S_{K}^{M}(L)=\Lambda^{1}_{K}(L)\times\bigg(-\frac{M}{2},\frac{M}{2}\bigg)^{n}
\end{equation*}
for $L,K,M\geq1$. Denote by $H_{\omega,S_{K}^{M}(L)}^{X,Y}$ the operator $H_{\omega}$ restricted to $L^{2}(S_{K}^{M}(L))$ with $X$ boundary conditions on $\partial\Lambda^{1}_{K}(L)\times\big(-\frac{M}{2},\frac{M}{2}\big)^{n}$ and $Y$ boundary conditions on $\Lambda^{1}_{K}(L)\times\partial\big(-\frac{M}{2},\frac{M}{2}\big)^{n}$, where $X,Y=D$ or $N$ refer to either Dirichlet or Neumann boundary conditions.

\begin{lem}\label{lemma-auxiliary}
Let $\eta<0$ and $L,K\geq1$. For a.e. $\omega\in\Omega$, there exist constants $\alpha>0$, $M_{0}>0$ and $C>0$ such that
\begin{equation}\label{eigenvalue-counting-fun-comparison}
N\Big(H_{\omega,S_{K}^{M}(L)}^{X,D},E\Big)\leq N\Big(H_{\omega,S_{K}^{1}(L)}^{X},E\Big)\leq N\Big(H_{\omega,S_{K}^{M}(L)}^{X,D},E+C(K+L-1)^{d_{1}}K^{d_{2}}e^{-\alpha M}\Big)
\end{equation}
for both $X=D$ and $X=N$, all $M\geq M_{0}$ and all $E\leq\eta$.
\end{lem}
\begin{proof}
The lemma follows from Theorem \ref{theorem-expo-decay-eigenfun} proven in Subsection \ref{subsec-decay-of-eigenfunction} and \cite[Lemma 2.9]{KW06}. See \cite[Lemma 2.5]{KW06} for the arguments.
\end{proof}

The next lemma gives an alternative representation of $N_{1,L}^{X}(E)$ for $E<0$.

\begin{lem}\label{lemma-alternative-representation}
Suppose $L\geq1$. For all but possible countably many $E<0$ and any $\rho_{1},\rho_{2}>0$, there holds \begin{equation}\label{IDSS-alternating-representation}
\lim_{K\rightarrow\infty}\frac{N\Big(H_{\omega,S_{K}^{M}(L)}^{X,D},E\Big)}{|\Lambda^{1}_{K}(L)|}=N_{1,L}^{X}(E)
\end{equation}
for a.e. $\omega\in\Omega$ and both $X=D$ and $X=N$, where we set $M=(K+L-1)^{\rho_{1}}K^{\rho_{2}}$.
\end{lem}
\begin{proof}
Let $\eta<0$. We claim that \eqref{IDSS-alternating-representation} holds for all but possible countably many $E\leq\eta$. On one hand, Theorem \ref{theorem-IDSS-cuboid} $\rm(i)$ and the first inequality in \eqref{eigenvalue-counting-fun-comparison} give
\begin{equation*}
N_{1,L}^{X}(E)\geq\limsup_{H\rightarrow\infty}\frac{N\Big(H_{\omega,S_{K}^{M}(L)}^{X,D},E\Big)}{|\Lambda_{K}^{1}(L)|}\quad\text{for all}\,\,E\leq\eta.
\end{equation*}
On the other hand, for $E\leq\eta$ and any $\epsilon>0$
\begin{equation*}
\begin{split}
N_{1,L}^{X}(E-\epsilon)&=\lim_{K\rightarrow\infty}\frac{N\Big(H_{\omega,S_{K}^{1}(L)}^{X},E-\epsilon\Big)}{|\Lambda_{K}^{1}(L)|}\\
&\leq\liminf_{K\rightarrow\infty}\frac{N\Big(H_{\omega,S_{K}^{1}(L)}^{X},E-C(K+L-1)^{d_{1}}K^{d_{2}}e^{-\alpha
M}\Big)}{|\Lambda_{K}^{1}(L)|}\\
&\leq\liminf_{K\rightarrow\infty}\frac{N\Big(H_{\omega,S_{K}^{M}(L)}^{X,D},E\Big)}{|\Lambda_{K}^{1}(L)|},
\end{split}
\end{equation*}
where we used Theorem \ref{theorem-IDSS-cuboid} $\rm(i)$, the fact that $N\big(H_{\omega,S_{K}^{1}(L)}^{X},E-\epsilon\big)\leq N\big(H_{\omega,S_{K}^{1}(L)}^{X},E-C(K+L-1)^{d_{1}}K^{d_{2}}e^{-\alpha M}\big)$ for all large enough $K$ and the second inequality in \eqref{eigenvalue-counting-fun-comparison}. Since $N_{2,L}^{X}(E)$ is continuous at all but possible countably many $E<0$, by letting
$\epsilon\rightarrow0$, we have for all but possible countably many $E\leq\eta$
\begin{equation*}
N_{1,L}^{X}(E)\leq\liminf_{K\rightarrow\infty}\frac{N\Big(H_{\omega,S_{K}^{M}(L)}^{X,D},E\Big)}{|\Lambda_{K}^{1}(L)|}.
\end{equation*}

To finish the proof, we set $\eta_{0}=-\infty$ and pick a strictly increasing sequence $\{\eta_{k}\}_{k=1}^{\infty}\subset(-\infty,0)$
such that $\eta_{k}\rightarrow0$ as $k\rightarrow\infty$. Then the above argument says that for any $k\in\mathbb{N}$, \eqref{IDSS-alternating-representation} holds for all but possible countably many $E\in(\eta_{0},\eta_{k}]$. In particular, for any $k\in\mathbb{N}$, \eqref{IDSS-alternating-representation} holds for all but possible countably many $E\in(\eta_{k-1},\eta_{k}]$. The result of the lemma then follows from the obvious fact $(-\infty,0)=\cup_{k=1}^{\infty}(\eta_{k-1},\eta_{k}]$.
\end{proof}

We proceed to prove the statement $\rm(ii)$ in Theorem \ref{theorem-IDSS-cuboid}.

\begin{proof}[Proof of Theorem \ref{theorem-IDSS-cuboid} \rm(ii)]
For $\eta<0$, Laplace transform estimate (see \cite[Theorem 3.3]{KM82}) gives
\begin{equation}\label{laplace-tran-estimate}
0\leq\int_{-\infty}^{\eta}\bigg(N\Big(H_{\omega,S_{K}^{M}(L)}^{N,D},E\Big)-N\Big(H_{\omega,S_{K}^{M}(L)}^{D,D},E\Big)\bigg)dE\leq\frac{e^{t\eta}}{t}\bigg({\rm
Tr}\bigg[e^{-tH_{\omega,S_{K}^{M}(L)}^{N,D}}-e^{-tH_{\omega,S_{K}^{M}(L)}^{D,D}}\bigg]\bigg)
\end{equation}
for any $t>0$. By positivity of operators, the fact $V_{\omega}\geq V$ and the H\"{o}lder's inequality for trace ideals with conjugate exponents $p_{0}$ and $q_{0}$ ($1<p_{0},q_{0}<\infty$), we estimate
\begin{equation}\label{trace-estimate}
\begin{split}
&{\rm Tr}\bigg[e^{-H_{\omega,S_{K}^{M}(L)}^{N,D}}-e^{-H_{\omega,S_{K}^{M}(L)}^{D,D}}\bigg]\\
&\quad\quad\leq{\rm Tr}\bigg[\bigg(e^{\Delta_{S_{K}^{M}(L)}^{N,D}}-e^{\Delta_{S_{K}^{M}(L)}^{D,D}}\bigg)^{\frac{1}{p_{0}}}\bigg(e^{\Delta_{S_{K}^{M}(L)}^{N,D}}-e^{\Delta_{S_{K}^{M}(L)}^{D,D}}\bigg)^{\frac{1}{q_{0}}}e^{-V_{0}-V}\bigg]\\
&\quad\quad\leq\bigg({\rm Tr}\bigg[e^{\Delta_{S_{K}^{M}(L)}^{N,D}}-e^{\Delta_{S_{K}^{M}(L)}^{D,D}}\bigg]\bigg)^{\frac{1}{p_{0}}}\bigg({\rm Tr}\bigg[\bigg(e^{\Delta_{S_{K}^{M}(L)}^{N,D}}-e^{\Delta_{S_{K}^{M}(L)}^{D,D}}\bigg)e^{-q_{0}(V_{0}+V)}\bigg]\bigg)^{\frac{1}{q_{0}}}\\
&\quad\quad\leq\bigg({\rm Tr}\bigg[e^{\Delta_{S_{K}^{M}(L)}^{N,D}}-e^{\Delta_{S_{K}^{M}(L)}^{D,D}}\bigg]\bigg)^{\frac{1}{p_{0}}}\bigg({\rm Tr}\bigg[e^{\Delta_{S_{K}^{M}(L)}^{N,D}-q_{0}(V_{0}+V)}\bigg]\bigg)^{\frac{1}{q_{0}}},
\end{split}
\end{equation}
where $\Delta_{S_{K}^{M}(L)}^{X,Y}$ is the Laplacian $\Delta$ restricted to $L^{2}(S_{K}^{M}(L))$ with $X$ boundary conditions on $\partial\Lambda_{K}(L)\times\big(-\frac{M}{2},\frac{M}{2}\big)^{n}$ and $Y$ boundary conditions on $\Lambda_{K}(L)\times\partial\big(-\frac{M}{2},\frac{M}{2}\big)^{n}$. Set $U=V_{0}+V$ and $H(q_{0}U)=-\Delta+q_{0}U$ and denote by $H_{S_{K}^{M}(L)}^{X,Y}(q_{0}U)$ the corresponding localized operator with obvious meaning. By setting $t=1$ in \eqref{laplace-tran-estimate}, \eqref{trace-estimate} yields
\begin{equation}\label{sandwich-bound-uniqueness}
\begin{split}
0&\leq\frac{1}{|\Lambda_{K}^{1}(L)|}\int_{-\infty}^{\eta}\bigg(N\Big(H_{\omega,S_{K}^{M}(L)}^{N,D},E\Big)-N\Big(H_{\omega,S_{K}^{M}(L)}^{D,D},E\Big)\bigg)dE\\
&\quad\quad\leq\frac{e^{\eta}}{|\Lambda_{K}^{1}(L)|}\bigg({\rm Tr}\bigg[e^{\Delta_{S_{K}^{M}(L)}^{N,D}}-e^{\Delta_{S_{K}^{M}(L)}^{D,D}}\bigg]\bigg)^{\frac{1}{p_{0}}}\bigg({\rm Tr}\bigg[e^{-H_{S_{K}^{M}(L)}^{N,D}(q_{0}U)}\bigg]\bigg)^{\frac{1}{q_{0}}}.
\end{split}
\end{equation}

For the first trace in the last line of
\eqref{sandwich-bound-uniqueness}, we employ Lemma
\ref{lemma-app-difference-of-trace} and thus obtain
\begin{equation}\label{trace-estimate-1}
\begin{split}
&{\rm
Tr}\bigg[e^{\Delta_{S_{K}^{M}(L)}^{N,D}}-e^{\Delta_{S_{K}^{M}(L)}^{D,D}}\bigg]\\
&\quad\quad\leq\bigg[\bigg(1+\frac{K+L-1}{\sqrt{4\pi}}\bigg)^{d_{1}}\bigg(1+\frac{K}{\sqrt{4\pi}}\bigg)^{d_{2}}-\bigg(\frac{K+L-1}{\sqrt{4\pi}}-1\bigg)^{d_{1}}\bigg(\frac{K}{\sqrt{4\pi}}-1\bigg)^{d_{2}}\bigg]\times\bigg(\frac{M}{\sqrt{4\pi}}\bigg)^{n}\\
&\quad\quad\leq
C_{1}\Big((K+L-1)^{d_{1}-1}K^{d_{2}}+(K+L-1)^{d_{1}}K^{d_{2}-1}\Big)M^{n}
\end{split}
\end{equation}
for some $C_{1}=C_{1}(d_{1},d_{2},n)>0$.

For the second trace in the last line of \eqref{sandwich-bound-uniqueness}, we use the fact that $q_{0}U\in\mathcal{K}(\mathbb{R}^{d+n})$, which implies that $q_{0}U$ is infinitesimally form bounded with respect to $-\Delta_{S_{K}^{M}(L)}^{N,D}$. It then follows from min-max principle that for any $\epsilon>0$, there exists $C_{\epsilon}\geq0$ such that
\begin{equation}\label{estimate-on-eigenvalues}
E_{M,N}\Big(H_{S_{K}^{M}(L)}^{N,D}(q_{0}U)\Big)\geq(1-\epsilon)E_{M,N}\Big(-\Delta_{S_{K}^{M}(L)}^{N,D}\Big)-C_{\epsilon},\quad M\in\mathbb{N}_{0}^{d},\,\, N\in\mathbb{N}^{d},
\end{equation}
where $E_{M,N}\Big(-\Delta_{S_{K}^{M}(L)}^{N,D}\Big)$, $M\in\mathbb{N}_{0}^{d}$, $N\in\mathbb{N}^{d}$, given in \eqref{eigenvalue-N-D}, are eigenvalues of
$-\Delta_{S_{K}^{M}(L)}^{N,D}$, and $E_{M,N}\Big(H_{S_{K}^{M}(L)}^{N,D}(q_{0}U)\Big)$, $M\in\mathbb{N}_{0}^{d}$, $N\in\mathbb{N}^{d}$, are eigenvalues of
$H_{S_{K}^{M}(L)}^{N,D}(q_{0}U)$. By means of \eqref{estimate-on-eigenvalues} with fixed $\epsilon\in(0,1)$, \eqref{eigenvalue-N-D} and arguments as in the proof of Lemma \ref{lemma-app-difference-of-trace}, we deduce
\begin{equation}\label{trace-estimate-2}
\begin{split}
{\rm Tr}\bigg[e^{-H_{S_{K}^{M}(L)}^{N,D}(q_{0}U)}\bigg]&=\sum_{M\in\mathbb{N}_{0}^{d}}\sum_{N\in\mathbb{N}^{d}}{\rm
exp}\Big\{-E_{M,N}\Big(H_{S_{K}^{M}(L)}^{N,D}(q_{0}U)\Big)\Big\}\\
&\leq e^{C_{\epsilon}}\sum_{M\in\mathbb{N}_{0}^{d}}\sum_{N\in\mathbb{N}^{d}}{\rm exp}\Big\{(\epsilon-1)E_{M,N}\Big(-\Delta_{S_{K}^{M}(L)}^{N,D}\Big)\Big\}\\
&\leq e^{C_{\epsilon}}\bigg(1+\frac{K+L-1}{\sqrt{4\pi(1-\epsilon)}}\bigg)^{d_{1}}\bigg(1+\frac{K}{\sqrt{4\pi(1-\epsilon)}}\bigg)^{d_{2}}\bigg(\frac{M}{\sqrt{4\pi(1-\epsilon)}}\bigg)^{n}\\
&\leq C_{2}(K+L-1)^{d_{1}}K^{d_{2}}M^{n}
\end{split}
\end{equation}
for some $C_{2}=C_{2}(d_{1},d_{2},n)>0$.

By estimates \eqref{trace-estimate-1} and \eqref{trace-estimate-2}, and taking $M=(K+L-1)^{\rho_{1}}K^{\rho_{2}}$ with $\rho_{1}+\rho_{2}<\frac{1}{np_{0}}$, the term in the last line of \eqref{sandwich-bound-uniqueness} is bounded from above by
\begin{equation*}
\begin{split}
&\frac{e^{\eta}C_{1}^{\frac{1}{p_{0}}}C_{2}^{\frac{1}{q_{0}}}\Big((K+L-1)^{d_{1}-1}K^{d_{2}}+(K+L-1)^{d_{1}}K^{d_{2}-1}\Big)^{\frac{1}{p_{0}}}(K+L-1)^{\frac{d_{1}}{q_{0}}+\rho_{1}n}K^{\frac{d_{2}}{q_{0}}+\rho_{2}n}}{(K+L-1)^{d_{1}}K^{d_{2}}}\\
&\quad\quad\leq C_{3}K^{(\rho_{1}+\rho_{2})n-\frac{1}{p_{0}}}\rightarrow0\quad\text{as}\quad K\rightarrow\infty,
\end{split}
\end{equation*}
where $C_{3}=C_{3}(d_{1},d_{2},n,L)>0$. This is to say
\begin{equation*}
\lim_{K\rightarrow\infty}\frac{1}{|\Lambda_{K}^{1}(L)|}\int_{-\infty}^{\eta}\bigg(N\Big(H_{\omega,S_{K}^{M}(L)}^{N,D},E\Big)-N\Big(H_{\omega,S_{K}^{M}(L)}^{D,D},E\Big)\bigg)dE=0,
\end{equation*}
which, together with Lemma \ref{lemma-alternative-representation} and the fact that $N_{1,L}^{X}(E)$ is continuous at all but possible countably many $E<0$, implies that $N_{1,L}^{D}(E)=N_{1,L}^{N}(E)$ for all but possible countably many $E\leq\eta$.

The result is then a simple consequence of the arguments as in the last paragraph of the proof of Lemma \ref{lemma-alternative-representation}.
\end{proof}

Theorem \ref{theorem-IDSS-cuboid} says that $N_{1,L}(E)$, defined to be the common values of $N_{1,L}^{D}(E)$ and $N_{1,L}^{N}(E)$, is well-defined for all but possible countably many $E<0$. Moreover, Lemma \ref{lemma-alternative-representation} and Theorem \ref{theorem-IDSS-cuboid} $\rm(ii)$ say that for all but possible countably many $E<0$ and any $\rho_{1},\rho_{2}>0$, there holds
\begin{equation}\label{IDSS-alternating-representation-1}
\lim_{K\rightarrow\infty}\frac{N\Big(H_{\omega,S_{K}^{M}(L)}^{X,D},E\Big)}{|\Lambda_{K}(L)|}=N_{1,L}(E)
\end{equation}
for a.e. $\omega\in\Omega$ and both $X=D$ and $X=N$, where $M=(K+L-1)^{\rho_{1}}K^{\rho_{2}}$.

Theorem \ref{theorem-IDSS-cuboid} is only part of Theorem \ref{theorem-IDSS} $\rm(ii)$. We now prove the remaining part of Theorem \ref{theorem-IDSS} $\rm(ii)$ and Theorem \ref{theorem-IDSS} $\rm(iii)$.

\begin{thm}
There holds $N(E)=N_{k,L}(E)$ for $k=\{1,2\}$, all $L\geq1$ and all possible countably many $E<0$.
\end{thm}
\begin{proof}
We focus on the case $k=1$. Pick any $L_{1}, L_{2}\in[1,\infty)$ with $L_{1}<L_{2}$. We claim that $N(E)=N_{1,L}(E)$ for all $L\in[L_{1},L_{2}]$ and all but possible countably many $E<0$. Let $\rho=\rho_{1}+\rho_{2}$ with $\rho_{1}$ and $\rho_{2}$ being the same as in the proof of Lemma \ref{lemma-alternative-representation} and set
\begin{equation*}
S_{K}^{K^{\rho}}=\Lambda_{K}\times\bigg(-\frac{K^{\rho}}{2},\frac{K^{\rho}}{2}\bigg)^{n}.
\end{equation*}
Denote by $H_{\omega,S_{K}^{K^{\rho}}}^{X,Y}$ the operator $H_{\omega}$ restricted to $L^{2}(S_{K}^{K^{\rho}})$ with $X$ boundary conditions on
$\partial\Lambda_{K}\times\big(-\frac{K^{\rho}}{2},\frac{K^{\rho}}{2}\big)^{n}$ and $Y$ boundary conditions on $\Lambda_{K}\times\partial\big(-\frac{K^{\rho}}{2},\frac{K^{\rho}}{2}\big)^{n}$. Since $S_{K}^{K^{\rho}}\subset S_{K}^{M}(L)\subset S_{K+L}^{(K+L)^{\rho}}$, where $M=(K+L-1)^{\rho_{1}}K^{\rho_{2}}$, Dirichlet-Neumann bracketing ensures that for $E<0$
\begin{equation*}
N\Big(H_{\omega,S_{K}^{H^{\rho}}}^{D,D},E\Big)\leq N\Big(H_{\omega,S_{K}^{M}(L)}^{D,D},E\Big)\leq N\Big(H_{\omega,S_{K+L}^{(K+L)^{\rho}}}^{D,D},E\Big),
\end{equation*}
which implies that
\begin{equation}\label{sandwich}
\frac{N\Big(H_{\omega,S_{K}^{K^{\rho}}}^{D,D},E\Big)}{|\Lambda_{K}|}\leq \frac{|\Lambda_{K}(L)|}{|\Lambda_{K}|}\frac{N\Big(H_{\omega,S_{K}^{M}(L)}^{D,D},E\Big)}{|\Lambda_{K}(L)|}\leq
\frac{|\Lambda_{K+L}|}{|\Lambda_{K}|}\frac{N\Big(H_{\omega,S_{K+L}^{(K+L)^{\rho}}}^{D,D},E\Big)}{|\Lambda_{K+L}|}.
\end{equation}
It was proven in \cite{KW06} that $\lim_{K\rightarrow\infty}\frac{1}{|\Lambda_{K}|}N\Big(H_{\omega,S_{K}^{K^{\rho}}}^{D,D},E\Big)=N(E)$
for all but possible countably many $E<0$, and the limit $\lim_{K\rightarrow\infty}\frac{1}{{|\Lambda_{K+L}|}}N\Big(H_{\omega,S_{K+L}^{(K+L)^{\rho}}}^{D,D},E\Big)=N(E)$
holds for all $L\in[L_{1},L_{2}]$ and all possible countably many $E<0$. Moreover, the limit $\lim_{K\rightarrow\infty}\frac{1}{|\Lambda_{K}(L)|}N\Big(H_{\omega,S_{K}^{M}(L)}^{D,D},E\Big)=N_{1,L}(E)$ holds for all but possible countably many $E<0$ if we take $X=D$ in \eqref{IDSS-alternating-representation-1}. Therefore, passing to the limit $K\rightarrow\infty$ in \eqref{sandwich}, the claim
follows.

The result of the theorem is obtained by picking countably many compact intervals covering $[1,\infty)$.
\end{proof}

%%%%%%%%%%%%%%%%%%%%%%%%%%%%%%%%%%%%%%%%%%%%%%%%%%%%%%%%%%%%%%%%%%%%%%%%%%%%%%%%%%%%%%%%%%%%%%%%%%%%%%%%%%%%%%%%%%%%%%%%%%%%

\subsection{Partially Exponential Decay of Eigenfunctions}\label{subsec-decay-of-eigenfunction}

In the proof of Lemma \ref{lemma-auxiliary}, we employed Theorem \ref{theorem-expo-decay-eigenfun}, which is the purpose of this section.

Let $\Lambda_{L,K}=\big(-\frac{L}{2},\frac{L}{2}\big)^{d_{1}}\times\big(-\frac{K}{2},\frac{K}{2}\big)^{d_{2}}$ and set $S_{L,K}=\Lambda_{L,K}\times\mathbb{R}^{n}$ for $L,K\geq1$. The main result in this subsection is the following theorem about the exponential decay of eigenfunctions, corresponding to negative eigenvalues of $H_{\omega,S_{L,K}}^{X}$, in the $y$-direction.

\begin{thm}\label{theorem-expo-decay-eigenfun} Let $\eta<0$. Then, for a.e. $\omega\in\Omega$, there holds the following statement: there exist $C=C(\eta)>0$ and $\gamma=\gamma(\eta)>0$ such that for both $X=D$ and $N$, any $L,K\geq1$ and any $L^{2}(\Lambda_{L,K})$-normalized eigenfunction
$\psi_{E}$ of $H_{\omega,S_{L,K}}^{X}$ corresponding to an eigenvalue $E\leq\eta$, one has
\begin{equation*}
\sup_{x\in\Lambda_{L,K}}|\psi_{E}(x,y)|\leq Ce^{-\gamma|y|}
\end{equation*}
for large enough $|y|$.
\end{thm}

To prove the above theorem, we first prove several lemmas. The first one gives an estimate related to the integral kernel of $e^{t\Delta_{S_{L,K}}^{N}}$.

\begin{lem}\label{lemma-heat-kernel-estimate}
Let $L,K\geq1$. The integral kernel $e^{t\Delta_{S_{L,K}}^{N}}(\cdot,\cdot)$ of $e^{t\Delta_{S_{L,K}}^{N}}$ satisfies
\begin{equation*}
\int_{S_{L,K}}\big|e^{t\Delta_{S_{L,K}}^{N}}(x,y,\bar{x},\bar{y})\big|^{2}d\bar{x}d\bar{y}\leq\bigg(\frac{2}{L}+\frac{1}{\sqrt{2\pi t}}\bigg)^{d_{1}}\bigg(\frac{2}{K}+\frac{1}{\sqrt{2\pi t}}\bigg)^{d_{2}}2^{-3n/2}(\pi t)^{-n/2}
\end{equation*}
for all $(x,y)\in S_{L,K}$ and all $t>0$.
\end{lem}
\begin{proof}
Note that $\Delta_{S_{L,K}}^{N}=\Delta_{\Lambda_{L,K}}^{N}\otimes I_{\mathbb{R}^{n}}+I_{\Lambda_{L,K}}\otimes\Delta_{\mathbb{R}^{n}}$, where $\Delta_{\Lambda_{L,K}}^{N}$ is the Neumann Laplacian on $L^{2}(\Lambda_{L,K})$, $I_{\mathbb{R}^{n}}$ is the identity operator on $L^{2}(\mathbb{R}^{n})$, $I_{\Lambda_{L,K}}$ is the identity operator on $L^{2}(\Lambda_{L,K})$ and $\Delta_{\mathbb{R}^{n}}$ is the Laplacian $\Delta$ on $L^{2}(\mathbb{R}^{n})$. It follows that $e^{t\Delta_{S_{L,K}}^{N}}=e^{t\Delta_{\Lambda_{L,K}}^{N}}\otimes e^{t\Delta_{\mathbb{R}^{n}}}$, which implies that
\begin{equation}\label{integral-kernel-identity}
e^{t\Delta_{S_{L,K}}^{N}}(x,y,\bar{x},\bar{y})=e^{t\Delta_{\Lambda_{L,K}}^{N}}(x,\bar{x})e^{t\Delta_{\mathbb{R}^{n}}}(y,\bar{y})\quad\text{for}\,\,
(x,y),\,\,(\bar{x},\bar{y})\in S_{L,K}.
\end{equation}
By \eqref{integral-kernel-identity}, Lemma \ref{lemma-app-heat-kernel-estimate-1} and the heat kernel $e^{t\Delta_{\mathbb{R}^{n}}}(y,\bar{y})=(4\pi
t)^{-n/2}e^{-|y-\bar{y}|^{2}/4t}$, $y,\bar{y}\in\mathbb{R}^{n}$, we estimate
\begin{equation*}
\begin{split}
&\int_{S_{L,K}}\big|e^{t\Delta_{S_{L,K}}^{N}}(x,y,\bar{x},\bar{y})\big|^{2}d\bar{x}d\bar{y}\\
&\quad\quad=\int_{\Lambda_{L,K}}\big|e^{t\Delta_{\Lambda_{L,K}}^{N}}(x,\bar{x})\big|^{2}d\bar{x}\int_{\mathbb{R}^{n}}\big|e^{t\Delta_{\mathbb{R}^{n}}}(y,\bar{y})\big|^{2}d\bar{y}\\
&\quad\quad\leq\bigg(\frac{2}{L}+\frac{1}{\sqrt{2\pi t}}\bigg)^{d_{1}}\bigg(\frac{2}{K}+\frac{1}{\sqrt{2\pi t}}\bigg)^{d_{2}}2^{-3n/2}(\pi t)^{-n/2}.
\end{split}
\end{equation*}
This completes the proof.
\end{proof}

The next lemma gives a general result of the boundedness of the semigroup generated by $-H_{S_{L,K}}^{X}(W)$, where $H_{S_{L,K}}^{X}(W)$ is the operator $H(W)=-\Delta+W$ restricted to $L^{S_{L,K}}$ with $X$ boundary condition on $\partial S_{L,K}$.

\begin{lem}\label{lemma-boundedness-semigroup}
Let $W:\mathbb{R}^{d+n}\rightarrow\mathbb{R}$ be such that the positive part $W_{+}\in\mathcal{K}_{\rm loc}(\mathbb{R}^{d+n})$ and the negative part $W_{-}\in\mathcal{K}(\mathbb{R}^{d+n})$. Let $H(W)=-\Delta+W$. Then there exists some $C>0$ such that
\begin{equation*}
\big\|{\rm exp}\big\{-tH_{S_{L,K}}^{X}(W)\big\}\big\|_{2,\infty}\leq Ce^{-t\inf\sigma(H(W))}
\end{equation*}
for all $L,K\geq1$, $t\geq1$ and both $X=D$ and $N$, where $\|\cdot\|_{2,\infty}$ is the operator norm of a bounded linear operator from $L^{2}$ to $L^{\infty}$.
\end{lem}

\begin{rem}\label{remark-boundedness-semigroup}
We will use Lemma \ref{lemma-boundedness-semigroup} in the cases that $W=V_{0}$ and $W=2(V_{0}+V)$. In the case $W=V_{0}$, assumption $\rm(H1)$ says $\inf\sigma(H(W))=0$. Under the assumption of Lemma \ref{lemma-boundedness-semigroup}, we have
\begin{equation*}
\big\|{\rm exp}\big\{-tH_{S_{L,K}}^{X}(W)\big\}\big\|_{1,\infty}\leq Ce^{-t\inf\sigma(H(W))}
\end{equation*}
for all $L,K\geq1$, $t\geq1$ and both $X=D$ and $N$, which is a simple consequence of the semigroup property and duality.
\end{rem}

\begin{proof}[Proof of Lemma \ref{lemma-boundedness-semigroup}]
The lemma in the case $X=D$ is well-known in the theory of Schr\"{o}dinger operators (see e.g \cite[Eq.(2.40)]{BHL00}). We prove the lemma in the case $X=N$.

By semigroup property and the fact that $\inf\sigma\big(H_{S_{L,K}}^{N}(W)\big)\geq\inf\sigma(H(W))$, we have for any fixed $\tau\in(0,1)$ and $t\geq1$
\begin{equation}\label{estimate-13}
\begin{split}
&\big\|{\rm exp}\big\{-tH_{S_{L,K}}^{N}(W)\big\}\big\|_{2,\infty}\\
&\quad\quad\leq\big\|{\rm exp}\big\{-(t-\tau)H_{S_{L,K}}^{N}(W)\big\}\big\|_{2,2}\big\|{\rm exp}\big\{-\tau H_{S_{L,K}}^{N}(W)\big\}\big\|_{2,\infty}\\
&\quad\quad\leq e^{-(t-\tau)\inf\sigma\big(H_{S_{L,K}}^{N}(W)\big)}\big\|{\rm exp}\big\{-\tau H_{S_{L,K}}^{N}(W)\big\}\big\|_{2,\infty}\\
&\quad\quad\leq e^{-(t-\tau)\inf\sigma(H(W))}\big\|{\rm exp}\big\{-\tau H_{S_{L,K}}^{N}(W)\big\}\big\|_{2,\infty}.
\end{split}
\end{equation}
To estimate the term $\big\|{\rm exp}\big\{-\tau H_{S_{L,K}}^{N}(W)\big\}\big\|_{2,\infty}$, we argue as follows. For any $\psi\in L^{2}(S_{L,K})$, H\"{o}lder's inequality and Lemma \ref{lemma-heat-kernel-estimate} yield
\begin{equation*}
\begin{split}
\big|\big(e^{s\Delta_{S_{L,K}}^{N}}\psi\big)(x,y)\big|\leq\int_{S_{L,K}}\big|e^{s\Delta_{S_{L,K}}^{N}}(x,y,\bar{x},\bar{y})\big|\cdot|\psi(\bar{x},\bar{y})|d\bar{x}d\bar{y}\leq
c_{1}s^{-(d+n)/4}\|\psi\|_{2},
\end{split}
\end{equation*}
for $s\in(0,1]$, where $c_{1}=c_{1}(d_{1},d_{2},n)>0$. Since the above estimate holds for all $(x,y)\in S_{L,K}$, we obtain
\begin{equation*}
\big\|e^{s\Delta_{S_{L,K}}^{N}}\psi\big\|_{\infty}\leq c_{1}s^{-(d+n)/4}\|\psi\|_{2},\quad s\in(0,1],\,\,\psi\in
L^{2}(S_{L,K}),
\end{equation*}
which, by \cite[Corollary 2.4.7]{Da89}, is equivalent to
\begin{equation}\label{estimate-9}
\|\psi\|_{2}^{2+4/(d+n)}\leq c_{2}\big(\|\nabla\psi\|_{2}^{2}+\|\psi\|_{2}^{2}\big)\|\psi\|_{1}^{4/(d+n)},\quad 0\leq\psi\in H^{1}(S_{L,K})\cap L^{1}(S_{L,K})
\end{equation}
for some $c_{2}=c_{2}(d_{1},d_{2},n)>0$. Since the negative part of $V_{0}+V_{\omega}$ is infinitesimally form bounded with respect to the Neumann Laplacian, for any $\epsilon>0$ there is a constant $C_{\epsilon}>0$ (independent of $L$ and $H$) such that
\begin{equation}\label{estimate-8}
\langle\psi,H_{S_{L,K}}^{N}(W)\psi\rangle\geq(1-\epsilon)\|\nabla\psi\|_{2}^{2}-C_{\epsilon}\|\psi\|_{2}^{2},\quad\psi\in H^{1}(S_{L,K}),
\end{equation}
where $\langle\cdot,H_{S_{L,K}}^{N}(W)\cdot\rangle$ should be understood as the quadratic form associated with $H_{S_{L,K}}^{N}(W)$. Fix some $\epsilon\in(0,1)$ in \eqref{estimate-8}. For any $\psi$ satisfying $0\leq\psi\in H^{1}(S_{L,K})\cap L^{1}(S_{L,K})$, we plug \eqref{estimate-8} into
\eqref{estimate-9} to find
\begin{equation*}
\begin{split}
\|\psi\|_{2}^{2+4/(d+n)}&\leq\frac{c_{2}}{1-\epsilon}\big(\langle\psi,H_{S_{L,K}}^{N}(W)\psi\rangle+(C_{\epsilon}+1)\|\psi\|_{2}^{2}\big)\|\psi\|_{1}^{4/(d+n)}\\
&\leq c_{3}\big(\langle\psi,\big(H_{S_{L,K}}^{N}(W)-\inf\sigma(H(W))\big)\psi\rangle+\|\psi\|_{2}^{2}\big)\|\psi\|_{1}^{4/(d+n)}
\end{split}
\end{equation*}
for some $c_{3}=c_{3}(d_{1},d_{2},n)>0$. By the fact that $H_{S_{L,K}}^{N}(W)-\inf\sigma(H(W))\geq0$ and \cite[Corollary 2.4.7]{Da89}, this is equivalent to
\begin{equation}\label{estimate-10}
\big\|{\rm exp}\big\{-s\big(H_{S_{L,K}}^{N}(W)-\inf\sigma(H(W))\big)\big\}\psi\big\|_{\infty}\leq c_{4}s^{-(d+n)/4}\|\psi\|_{2},\,\, s\in(0,1],\,\,\psi\in
L^{2}(S_{L,K})
\end{equation}
for some $c_{4}=c_{4}(d_{1},d_{2},n)>0$. Setting $s=\tau$ in \eqref{estimate-10}, we obtain
\begin{equation*}
\big\|{\rm exp}\big\{-\tau H_{S_{L,H}}^{N}(W)\big\}\big\|_{2,\infty}\leq c_{4}e^{-\tau\inf\sigma(H(W))}\tau^{-(d+n)/4},
\end{equation*}
which together with \eqref{estimate-13} implies that
\begin{equation*}
\big\|{\rm exp}\big\{-tH_{S_{L,K}}^{N}(W)\big\}\big\|_{2,\infty}\leq c_{4}\tau^{-(d+n)/4}e^{-t\inf\sigma(H(W))}.
\end{equation*}
Since $\tau\in(0,1)$ is fixed, the lemma follows.
\end{proof}

Finally, we prove Theorem \ref{theorem-expo-decay-eigenfun}.

\begin{proof}[Proof of Theorem \ref{theorem-expo-decay-eigenfun}]
We sketch the proof since it basically repeats the proof of \cite[Theorem 2.2]{KW06}. By the fact $\psi_{E}={\rm exp}\big\{-t(H_{\omega,S_{L,K}}^{X}-E)\big\}\psi_{E}$ and the Feynman-Kac formula, we find
\begin{equation}\label{estimate-11}
|\psi_{E}(x,y)|\leq\int{\rm exp}\bigg\{\int_{0}^{t}\big(E-V_{0}(b^{X}(s))-V_{\omega}(b^{X}(s))\big)ds\bigg\}|\psi_{E}(b^{X}(t))|d\mathbb{P}_{x,y}^{X}(b^{X}),
\end{equation}
where $b^{X}:[0,\infty)\rightarrow S_{L,H}$ is the Brownian path starting at $(x,y)\in S_{L,H}$ with absorbing boundary conditions in the case $X=D$ (see e.g. \cite{Si05}) or reflecting boundary conditions at $\partial S_{L,H}$ in the case $X=N$ (see e.g. \cite{BR97}), and $\mathbb{P}_{x,y}^{X}$ is the corresponding Wiener measure.

Since
$V_{\omega}(x,y)\geq V(x,y)\geq\inf_{x\in\mathbb{R}^{d}}V(x,y)\rightarrow0$ as $|y|\rightarrow\infty$ by assumption $\rm (H3)$, we have $E-V_{\omega}(x,y)\leq E-\inf_{x\in\mathbb{R}^{d}}V(x,y)\leq\frac{\eta}{2}$ for $|y|$ large enough. We write $b^{X}=(b_{1}^{X},b_{2}^{X})$ with $b_{1}^{X}:[0,\infty)\rightarrow\Lambda_{L,H}$ and $b_{2}^{X}:[0,\infty)\rightarrow\mathbb{R}^{n}$ and let $\Omega_{t}=\big\{b^{X}\big|\sup_{s\in[0,t]}|b_{2}^{X}(s)-y|<\frac{|y|}{2}\big\}$. Taking $|y|$ large so that $\sup_{s\in[0,t]}|b_{2}^{X}(s)|$ is large
enough if $b^{X}\in\Omega_{t}$ and splitting the Wiener integral in \eqref{estimate-11} into integrals over $\Omega_{t}$ and its complement $\Omega_{t}^{c}$, we obtain
\begin{equation}\label{estimate-12}
\begin{split}
|\psi_{E}(x,y)|\leq&e^{\eta t/2}\int_{\Omega_{t}}{\rm exp}\bigg\{-\int_{0}^{t}V_{0}(b^{X}(s))ds\bigg\}|\psi_{E}(b^{X}(t))|d\mathbb{P}_{x,y}^{X}(b^{X})\\
&+\int_{\Omega_{t}^{c}}{\rm exp}\bigg\{\int_{0}^{t}\big(E-U(b^{X}(s)))\big)ds\bigg\}|\psi_{E}(b^{X}(t))|d\mathbb{P}_{x,y}^{X}(b^{X}),
\end{split}
\end{equation}
where we used the fact $V_{\omega}\geq V$ and introduced the notation $U=V_{0}+V$.

Let $H(V_{0})=-\Delta+V_{0}$. The first term on the right hand side of \eqref{estimate-12} is bounded from above by
\begin{equation*}
\begin{split}
&e^{\eta t/2}\int{\rm exp}\bigg\{-\int_{0}^{t}V_{0}(b^{X}(s))ds\bigg\}|\psi_{E}(b^{X}(t))|d\mathbb{P}_{x,y}^{X}(b^{X})\\
&\quad\quad=e^{\eta t/2}\Big(e^{-tH_{S_{L,K}}^{X}(V_{0})}|\psi_{E}|\Big)(x,y)\leq e^{\eta t/2}\Big\|e^{-tH_{S_{L,K}}^{X}(V_{0})}\Big\|_{2,\infty}\|\psi_{E}\|_{2}\leq C_{1}e^{\eta t/2}
\end{split}
\end{equation*}
by Lemma \ref{lemma-boundedness-semigroup} and Remark \ref{remark-boundedness-semigroup}.

Let $H(2U)=-\Delta+2U$. The second term on the right hand side of \eqref{estimate-12} is bounded from above by
\begin{equation*}
\begin{split}
&\int_{\Omega_{t}^{c}}{\rm exp}\bigg\{-\int_{0}^{t}U(b^{X}(s))ds\bigg\}|\psi_{E}(b^{X}(t))|d\mathbb{P}_{x,y}^{X}(b^{X})\\
&\quad\quad\leq(\mathbb{P}_{x,y}^{X}(\Omega_{t}^{c}))^{1/2}\bigg(\int{\rm exp}\bigg\{-\int_{0}^{t}2U(b^{X}(s))ds\bigg\}|\psi_{E}(b^{X}(t))|^{2}d\mathbb{P}_{x,y}^{X}(b^{X})\bigg)^{\frac{1}{2}}\\
&\quad\quad=(\mathbb{P}_{x,y}^{X}(\Omega_{t}^{c}))^{1/2}\bigg(\Big(e^{-tH_{S_{L,K}}^{X}(2U)}|\psi_{E}|^{2}\Big)(x,y)\bigg)^{\frac{1}{2}}\\
&\quad\quad\leq(\mathbb{P}_{x,y}^{X}(\Omega_{t}^{c}))^{1/2}\Big\|e^{-tH_{S_{L,K}}^{X}(2U)}\Big\|_{1,\infty}^{\frac{1}{2}}\|\psi_{E}\|_{2}\leq C_{2}(\mathbb{P}_{x,y}^{X}(\Omega_{t}^{c}))^{1/2}e^{t|\inf\sigma(H(2U))|/2}
\end{split}
\end{equation*}
by Lemma \ref{lemma-boundedness-semigroup} and Remark \ref{remark-boundedness-semigroup}. For the term $\mathbb{P}_{x,y}^{X}(\Omega_{t}^{c})$, we have
$\mathbb{P}_{x,y}^{X}(\Omega_{t}^{c})\leq4e^{-|y|^{2}/32t}$ by Levy's maximal inequality (see e.g. \cite{Si05}).

Thus, we have shown
\begin{equation*}
|\psi_{E}(x,y)|\leq C_{1}e^{\eta t/2}+2C_{2}e^{-|y|^{2}/64t}e^{t|\inf\sigma(H(2U))|/2}
\end{equation*}
for $x\in\Lambda_{L,H}$, $y\in\mathbb{R}^{n}$ with $|y|$ being large enough and $t\geq1$. The result of the theorem is attained by setting $t=\frac{|y|}{8(|\inf\sigma(H(2U))|)^{1/2}}$ and $\gamma=\min\Big\{\frac{(|\inf\sigma(H(2U))|)^{1/2}}{16},-\frac{\eta}{16(|\inf\sigma(H(2U))|)^{1/2}}\Big\}$.
\end{proof}

%%%%%%%%%%%%%%%%%%%%%%%%%%%%%%%%%%%%%%%%%%%%%%%%%%%%%%%%%%%%%%%%%%%%%%%%%%%%%%%%%%%%%%%%%%%%%%%%%%%%%%%%%%%%%%%%%%%%%%%%%%%%%

\section{Spectral Gap Estimates}\label{sec-spectral-gap}

This section is devoted to the discussion of the crucial spectral gap estimates (Subsection \ref{subsec-Mezincescu-spectral-gap}) related to $H_{\rm per}$ defined in \eqref{periodic-operator}, and the sandwiching bound for the IDSS for negative energies (Subsection \ref{subsec-IDSS-neg-energy}). Throughout this section, assumptions $\rm(H1)$, $\rm(H2)$, $\rm(H3)$ and $\rm(H4)$ are assumed to be satisfied.

Let's begin with the ground state of $H_{\rm per}$. Let $\psi_{0}$ be the ground state of $H_{\rm per}$, i.e., $H_{\rm per}\psi_{0}=E_{0}\psi_{0}$. We remark that $\psi_{0}\notin L^{2}(\mathbb{R}^{d+n})$ and $\psi_{0}$ can be taken to be positive, $\mathbb{Z}^{d}$-periodic and continuously differentiable in a
neighborhood of $\partial S$, where $S=\Lambda\times\mathbb{R}^{n}$ with $\Lambda$ being any $d$-dimensional open cuboid. For later use, we assume that $\psi_{0}$ is $L^{2}(S_{1})$-normalized, i.e., $\int_{S_{1}}\psi_{0}(x,y)^{2}dxdy=1$, where $S_{1}=\Lambda_{1}\times\mathbb{R}^{n}$ with $\Lambda_{1}$ being the unit open cube in $\mathbb{R}^{d}$ centered at $0\in\mathbb{R}^{d}$. Another property of $\psi_{0}$ is stated in the following

\begin{lem}[\cite{KW06}]\label{lemma-ground-state}
Let $\bar{\psi}_{0}(y)=\int_{\Lambda_{1}}\psi_{0}(x,y)dx$, $y\in\mathbb{R}^{n}$. Then there are constants $C_{1},C_{2}>0$ such that
\begin{equation*}
C_{1}\bar{\psi}_{0}(y)\leq\psi_{0}(x,y)\leq C_{2}\bar{\psi}_{0}(y)\quad\text{for all}\,\,x\in\mathbb{R}^{d}\,\,\text{and}\,\,y\in\mathbb{R}^{n}.
\end{equation*}
\end{lem}

%%%%%%%%%%%%%%%%%%%%%%%%%%%%%%%%%%%%%%%%%%%%%%%%%%%%%%%%%%%%%%%%%%%%%%%%%%%%%%%%%%%%%%%%%%%%%%%%%%%%%%%%

\subsection{Mezincescu Boundary Conditions and Spectral Gap Estimates}\label{subsec-Mezincescu-spectral-gap}

Recall that $S=\Lambda\times\mathbb{R}^{n}$ with $\Lambda$ being any $d$-dimensional open cuboid. We define
\begin{equation*}
\chi_{S}(x,y)=-\frac{1}{\psi_{0}(x,y)}(\overrightarrow{n}\cdot\nabla)\psi_{0}(x,y)\quad\text{for}\quad(x,y)\in\partial S,
\end{equation*}
where $\overrightarrow{n}$ is the outer normal vector of $\partial S$. Let $H_{{\rm per},S}^{\chi}$ be the restriction of $H_{\rm per}$ to $L^{2}(S)$ with Mezincescu boundary condition
\begin{equation*}
(\overrightarrow{n}\cdot\nabla)\psi=-\chi_{S}\psi\quad\text{on}\quad\partial S.
\end{equation*}
It is referred to \cite{Me87,Mi02} for more discussions about Mezincescu boundary condition.
\begin{lem}\label{H-per-property}
There hold
\begin{itemize}
\item[\rm(i)] $\inf\sigma(H_{{\rm per},S}^{\chi})=\inf\sigma(H_{\rm per})=E_{0}$;

\item[\rm(ii)] $\psi_{0}$ restricted to $S$, denoted by $\psi_{0}|_{S}$, continues to be a ground state of $H_{{\rm per},S}^{\chi}$, i.e., $H_{{\rm per},S}^{\chi}\psi_{0}|_{S}=E_{0}\psi_{0}|_{S}$. Moreover, $\psi_{0}|_{S}\in L^{2}(S)$.
\end{itemize}
\end{lem}
\begin{proof}
For $\rm(i)$, we refer to \cite{Me87}. $\rm(ii)$ follows from $\rm(i)$ and the fact that $\psi_{0}|_{S_{L}}$ is positive and satisfies the eigenvalue equation and the boundary condition.
\end{proof}

To study Lifshitz tails in the quantum regime as well as in the classical regime, we need to consider special strips. More
precisely, we need $\Lambda_{L}=\big(-\frac{L}{2},\frac{L}{2}\big)^{d}$ and $S_{L}=\Lambda_{L}\times\mathbb{R}^{n}$ for $L\geq1$. The following result corresponding to the gap of the lowest two eigenvalues of $H_{{\rm per},S_{L}}^{\chi}$ plays a crucial role in the proof of the existence of Lifshitz tails.

\begin{lem}[\cite{KW06}]\label{spectral-gap-1}
There exists $C_{\rm per}>0$ such that
\begin{equation*}
E_{1}(H_{{\rm per},S_{L}}^{\chi})-E_{0}(H_{{\rm per},S_{L}}^{\chi})\geq\frac{C_{\rm per}}{L^{2}}
\end{equation*}
for large enough $L$.
\end{lem}

\begin{rem}
We remark that the fact $\Lambda_{L}$ is a cube is essential in the proof of Lemma \ref{spectral-gap-1} (see \cite[Theorem 3.5, Lemma 3.6]{KW06}), since
$E_{1}(-\Delta_{\Lambda_{L}}^{N})=\frac{\pi^{2}}{L^{2}}$ is used and plays an important role, where $-\Delta_{\Lambda_{L}}^{N}$ is the Neumann Laplacian on the cube $\Lambda_{L}$. Fortunately, similar result preserves if $\Lambda_{L}$ is a cuboid (see Appendix \ref{sec-neumann-laplacian}).
\end{rem}

For later use, we let $\psi_{L}$ be the positive, normalized ground state of $H_{{\rm per},S_{L}}^{\chi}$, that is, $\psi_{L}>0$ pointwise and $\psi_{L}\in L^{2}(S_{L})$ with $L^{2}(S_{L})$-norm $1$ and $H_{{\rm per},S_{L}}^{\chi}\psi_{L}=E_{0}\psi_{L}$. This is guaranteed by Lemma \ref{H-per-property}. Moreover, since $\psi_{0}$ is $L^{2}(S_{1})$-normalized, Lemma \ref{H-per-property} insures that
\begin{equation}\label{ground-states}
\psi_{L}=\frac{1}{L^{d/2}}\psi_{0}|_{S_{L}}.
\end{equation}

To study Lifshitz tails in the quantum-classical/classical-quantum regime, we need results analogous to Lemma \ref{spectral-gap-1}, but related to $H_{\rm per}$ restricted to other kinds of strips. To be more specific, for $d=d_{1}+d_{2}$ with $d_{1},d_{2}\in\mathbb{N}$ and $L\geq1$, we let
\begin{itemize}
\item[\rm(i)] $\Lambda_{L,1}=\big(-\frac{L}{2},\frac{L}{2}\big)^{d_{1}}\times\big(-\frac{1}{2},\frac{1}{2}\big)^{d_{2}}$,\quad$S_{L,1}=\Lambda_{L,1}\times\mathbb{R}^{n}$ (for the quantum-classical regime);

\item[\rm(ii)] $\Lambda_{1,L}=\big(-\frac{1}{2},\frac{1}{2}\big)^{d_{1}}\times\big(-\frac{L}{2},\frac{L}{2}\big)^{d_{2}}$,\quad$S_{1,L}=\Lambda_{1,L}\times\mathbb{R}^{n}$ (for the classical-quantum regime).
\end{itemize}
Thus, recalling the notations above Theorem \ref{theorem-IDSS}, there hold $\Lambda_{L,1}=\Lambda_{1}^{1}(L)$, $S_{L,1}=S_{1}^{1}(L)$, $\Lambda_{1,L}=\Lambda_{1}^{2}(L)$ and $S_{1,L}=S_{1}^{2}(L)$. We prove the following estimate of the spectral gap between the lowest two eigenvalues of $H_{{\rm per},S}^{\chi}$ for $S=S_{L,1}$ or $S_{1,L}$.

\begin{lem}\label{spectral-gap}
There exists a constant $C>0$ such that
\begin{equation*}
E_{1}(H_{{\rm per},S}^{\chi})-E_{0}(H_{{\rm per},S}^{\chi})\geq\frac{C}{L^{2}}
\end{equation*}
for large enough $L$, where $S=S_{L,1}$ or $S_{1,L}$.
\end{lem}
\begin{proof}
We prove the lemma in the case $S=S_{L,1}$. The lemma in the case $S=S_{1,L}$ can be proven in a similar way.

We first show that the lemma holds when $V_{0}+V$ is independent of $x$. If $V_{0}+V$ is independent of $x$, then the ground state of $H_{\rm per}$ is independent of $x$, and thus the Mezincescu and Neumann boundary conditions agree. It then follows that the eigenvalues of $H_{{\rm per},S_{L,1}}^{\chi}$ are given by the sum of the eigenvalues of the Neumann Laplacian $-\Delta_{\Lambda_{L,1}}^{N}$ on $L^{2}(\Lambda_{L,1})$ and the negative eigenvalues of $H_{\mathbb{R}^{n}}:=-\Delta+V_{0}+V$ on $L^{2}(\mathbb{R}^{n})$.

We claim that $E_{0}(H_{{\rm per},S_{L,1}}^{\chi})=E_{0}(H_{\mathbb{R}^{n}})$ and $E_{1}(H_{{\rm per},S_{L,1}}^{\chi})=E_{0}(H_{\mathbb{R}^{n}})+E_{1}(-\Delta_{\Lambda_{L,1}}^{N})$ for large enough $L$. This follows from the fact
$E_{1}(H_{\mathbb{R}^{n}})-E_{0}(H_{\mathbb{R}^{n}})\geq C_{1}$ for some $C_{1}>0$, $E_{0}(-\Delta_{\Lambda_{L,1}}^{N})=0$ and the fact
$E_{1}(-\Delta_{\Lambda_{L,1}}^{N})=\frac{\pi^{2}}{L^{2}}$ (see Appendix \ref{sec-neumann-laplacian}). Alternatively, we can use Kr\"{o}ger's result (see \cite{Kr92}) on upper bounds for Neumann eigenvalues, i.e.,
\begin{equation*}
E_{1}(-\Delta_{\Lambda_{L,1}}^{N})\leq\frac{d}{d+2}\frac{4\pi^{2}}{(C_{d}|\Lambda_{L,1}|)^{2/d}}\rightarrow0\quad\text{as}\quad L\rightarrow\infty,
\end{equation*}
where $C_{d}=\frac{\pi^{d/2}}{\Gamma(d/2+1)}$ and $|\Lambda_{L,1}|$ is the $d$-dimensional Lebesgue measure of $\Lambda_{L,1}$. Hence, we have
\begin{equation*}
E_{1}(H_{{\rm per},S_{L,1}}^{\chi})-E_{0}(H_{{\rm per},S_{L,1}}^{\chi})=E_{1}(-\Delta_{\Lambda_{L,1}}^{N})=\frac{\pi^{2}}{L^{2}},
\end{equation*}
that is, the lemma in the case that $V_{0}+V$ is independent of $x$ holds.

For the rest of the proof, we can employ the arguments in \cite{KW06} with obvious changes and thus we omit it here.
\end{proof}

For later use, we set
\begin{equation}\label{ground-states-1}
\psi_{L,1}=\frac{1}{L^{d_{1}/2}}\psi_{0}|_{S_{L,1}}
\end{equation}
and
\begin{equation*}
\psi_{1,L}=\frac{1}{L^{d_{2}/2}}\psi_{0}|_{S_{1,L}}.
\end{equation*}
They are the positive, normalized ground states of $H_{per}^{S_{L,1},\chi}$ and $H_{per}^{S_{1,L},\chi}$, respectively.

%%%%%%%%%%%%%%%%%%%%%%%%%%%%%%%%%%%%%%%%%%%%%%%%%%%%%%%%%%%%%%%%%%%%%%%%%%%%%%%%%%%%%%%%%%%%%%%%%%%%

\subsection{Sandwiching Bound}\label{subsec-IDSS-neg-energy}

As byproducts of the proof of Theorem \ref{theorem-IDSS}, Akcoglu-Krengel ergodic theorem (see \cite{AK81} or \cite[Theorem 3.1]{KM82}) says
\begin{equation}\label{equality-1}
\begin{split}
\lim_{L\rightarrow\infty}\frac{N\big(H_{\omega,S_{L}}^{D},E\big)}{|\Lambda_{L}|}&=\sup_{L\geq1}\frac{\mathbb{E}\big\{N\big(H_{\bullet,S_{L}}^{D},E\big)\big\}}{|\Lambda_{L}|},\\
\lim_{K\rightarrow\infty}\frac{N\Big(H_{\omega,S_{K}^{k}(L)}^{D},E\Big)}{|\Lambda_{K}^{k}(L)|}&=\sup_{K\geq1}\frac{\mathbb{E}\Big\{N\Big(H_{\bullet,S_{K}^{k}(L)}^{D},E\Big)\Big\}}{|\Lambda_{K}^{k}(L)|},\quad
L\geq1,\,\,k\in\{1,2\}.
\end{split}
\end{equation}
where $\mathbb{E}$ is the expectation with respect to the probability measure $\mathbb{P}$.

One of two goals in this paper is to investigate the asymptotic behavior of $N(E)$ near $E_{0}$, the bottom of the almost sure spectrum of $H_{\omega}$, $\omega\in\Omega$. This starts with the following sandwiching bound.

\begin{lem}\label{sandwich-bound}
There holds
\begin{equation*}
\frac{1}{|\Lambda|}\mathbb{P}\Big\{\omega\in\Omega\Big|E_{0}(H_{\omega,S}^{D})\leq E\Big\}\leq N(E)\leq\frac{1}{|\Lambda|}N\big(H_{per,S}^{\chi},E\big)\mathbb{P}\Big\{\omega\in\Omega\Big|E_{0}(H_{\omega,S}^{\chi})\leq E\Big\}
\end{equation*}
for all $E<0$, where the pair $(\Lambda,S)$ is taken to be $(\Lambda_{L},S_{L})$, $(\Lambda_{L,1},S_{L,1})$ or $(\Lambda_{1,L},S_{1,L})$ for $L\geq1$.
\end{lem}
\begin{proof}
The first inequality follows from \eqref{equality-1}. See \cite{Me87} for the second one.
\end{proof}

Note that the sandwiching bound in Lemma \ref{sandwich-bound} involves the term $N\big(H_{{\rm per},S}^{\chi},E\big)$ for $E<0$, the eigenvalue counting function of $H_{{\rm per},S}^{\chi}$ for negative energies. It is well-defined and the corresponding IDS for $H_{\rm per}$ has the so-called van-Hove singularity (see e.g \cite{KS87}) since $H_{\rm per}$ describes an ordered system.

%%%%%%%%%%%%%%%%%%%%%%%%%%%%%%%%%%%%%%%%%%%%%%%%%%%%%%%%%%%%%%%%%%%%%%%%%%%%%%%%%%%%%%%%%%%%%%%%%%%%%%%%%%%%%%%%%%%%%%%%%%%%%%%%

\section{Lifshitz Tails}\label{proof-of-Lifshitz-tails}

In this section, we prove Theorem \ref{main-theorem}. Assumptions $\rm(H1)$, $\rm(H2)$, $\rm(H3)$, $\rm(H4)$, $\rm(H5)$ and $\rm(H6)$ are always assumed to be satisfied. Our proof is based on a combination of ideas used in \cite{KW05} and \cite{KW06}.

By the definition of $H_{\rm per}$ (see \eqref{periodic-operator}), we can rewrite $H_{\omega}$ as
\begin{equation*}
H_{\omega}=H_{\rm per}+W_{\omega},
\end{equation*}
where
\begin{equation*}
W_{\omega}(x,y)=\sum_{i\in\mathbb{Z}^{d}}(\omega_{i}-\omega_{\min})f(x-i,y),\quad x\in\mathbb{R}^{d},\,\,y\in\mathbb{R}^{n}.
\end{equation*}
We note that $W_{\omega}$ is nonnegative since $\omega_{\min}=\inf\text{supp}\mathbb{P}_{0}$.

To fix the terminology, we give the following definition related to Lifshitz tails.
\begin{defn}
If the limit $\lim_{E\rightarrow E_{0}}\frac{\ln|\ln N(E)|}{\ln(E-E_{0})}$ exists and satisfies
\begin{equation*}
\lim_{E\downarrow E_{0}}\frac{\ln|\ln N(E)|}{\ln(E-E_{0})}=-\eta
\end{equation*}
for some $\eta\in\mathbb{R}$, then we call $\eta$ the Lifshitz exponent.
\end{defn}
Hence, the proof of our main results can be understood to derive an expression for the Lifshitz exponent, which can be done by estimating an upper bound as well as a lower bound, and is given in the following subsections.

%%%%%%%%%%%%%%%%%%%%%%%%%%%%%%%%%%%%%%%%%%%%%%%%%%%%%%%%%%%%%%%%%%%%%%%%%%%%%%%%%%%%%%%%%%%%%%%%%%%%%%%%%%%%%%%%%%%%%%%%%%%

\subsection{Lower Bound}

In this subsection, we prove an upper bound on the Lifshitz exponent such that a lower bound on the limit $\lim_{E\downarrow E_{0}}\frac{\ln|\ln N(E)|}{\ln(E-E_{0})}$, if exists, is obtained. To do so, we first estimate an upper bound on the lowest Dirichlet eigenvalue.

\begin{lem}\label{lemma-lower-bound}
There are constants $c_{1},c_{2}>0$ such that the ground state energy, $E_{0}(H_{\omega,S_{L}}^{D})$, of $H_{\omega,S_{L}}^{D}$ satisfies
\begin{equation*}
E_{0}(H_{\omega,S_{L}}^{D})\leq E_{0}+\frac{c_{1}}{L^{2}}+\frac{c_{2}}{L^{d}}\int_{\Lambda_{L}}\widehat{W}_{\omega}(x)dx
\end{equation*}
for all $\omega\in\Omega$ and $L\geq1$, where
\begin{equation*}
\widehat{W}_{\omega}(x)=\sum_{i\in\mathbb{Z}^{d}}\frac{\omega_{i}-\omega_{\min}}{(1+|x_{1}-i_{1}|)^{\alpha_{1}}+(1+|x_{2}-i_{2}|)^{\alpha_{2}}},\quad
x\in\mathbb{R}^{d},\,\,\omega\in\Omega.
\end{equation*}
\end{lem}
\begin{proof}
Let $\theta\in\mathcal{C}_{0}^{\infty}(\Lambda_{1})$ with $0\leq\theta(x)\leq1$ for all $x\in\Lambda_{1}$ and $\theta\equiv1$ on $\Lambda_{\frac{1}{2}}$, and define $\theta_{L}(x)=\theta(\frac{x}{L})$ for $x\in\Lambda_{L}$. Recall that $\psi_{L}$ is defined in \eqref{ground-states}. Using $\theta_{L}\psi_{L}$ (which is in the domain of $H_{\omega,S_{L}}^{D}$) as the variational function in the Rayleigh-Ritz principle or min-max principle, integration by parts and the eigenvalue equation $H_{{\rm per},S_{L}}^{\chi}\psi_{L}=E_{0}\psi_{L}$, we obtain
\begin{equation}\label{estimate-1}
\begin{split}
E_{0}(H_{\omega,S_{L}}^{D})\leq\frac{\big\langle\theta_{L}\psi_{L},H_{\omega,S_{L}}^{D}(\theta_{L}\psi_{L})\big\rangle}{\|\theta_{L}\psi_{L}\|^{2}}
=E_{0}+\frac{\|(\nabla\theta_{L})\psi_{L}\|^{2}}{\|\theta_{L}\psi_{L}\|^{2}}+\frac{\langle\theta_{L}\psi_{L},W_{\omega}\theta_{L}\psi_{L}\rangle}{\|\theta_{L}\psi_{L}\|^{2}}.
\end{split}
\end{equation}

Since $\theta\equiv1$ on $\Lambda_{\frac{L}{2}}$ and $\psi_{0}$ is $S_{1}$-normalized, we estimate $\|\theta_{L}\psi_{L}\|^{2}\geq2^{-d}$. For the term
$\|(\nabla\theta_{L})\psi_{L}\|^{2}$, direct calculation shows
\begin{equation*}
\begin{split}
\|(\nabla\theta_{L})\psi_{L}\|^{2}=\frac{1}{L^{d+2}}\int_{S_{L}}\Big|\nabla\theta\Big(\frac{x}{L}\Big)\Big|^{2}\psi_{0}(x,y)^{2}dxdy\leq\frac{C_{1}}{L^{2}}\int_{S_{1}}\psi_{0}(Lx,y)^{2}dxdy=\frac{C_{1}}{L^{2}}
\end{split}
\end{equation*}
for some $C_{1}>0$, where the inequality follows from the change of variable and the uniform boundedness of $\nabla\theta$, and the second equality is because of the change of variable and the $\mathbb{Z}^{d}$-periodicity of $\psi_{0}$. Therefore, \eqref{estimate-1} implies that
\begin{equation}\label{estimate-2}
E_{0}(H_{\omega,S_{L}}^{D})\leq E_{0}+\frac{C_{2}}{L^{2}}+\frac{C_{3}}{L^{d}}\int_{S_{L}}W_{\omega}(x,y)\psi_{0}(x,y)^{2}dxdy
\end{equation}
for some $C_{2},C_{3}>0$. For the integral on the right-hand side of \eqref{estimate-2}, we claim that
\begin{equation}\label{estimate-3}
\int_{S_{L}}W_{\omega}(x,y)\psi_{0}(x,y)^{2}dxdy\leq C_{4}\int_{\Lambda_{L}}\widehat{W}_{\omega}(x)dx
\end{equation}
for some $C_{4}>0$. In fact, assumption $\rm(H5)$ and Lemma \ref{lemma-ground-state} imply that $\int_{S_{L}}W_{\omega}(x,y)\psi_{0}(x,y)^{2}dxdy\leq
C_{5}\int_{\Lambda_{L}}\widehat{W}_{\omega}(x)dx\int_{\mathbb{R}^{n}}\bar{\psi}_{0}(y)^{2}dy$. The convergence of the second integral, i.e., $\int_{\mathbb{R}^{n}}\bar{\psi}_{0}(y)^{2}dy$, follows from Lemma
\ref{lemma-ground-state}. More precisely,
\begin{equation*}
1=\int_{\mathbb{R}^{n}}\int_{\Lambda_{1}}\psi_{0}(x,y)^{2}dxdy\geq C_{6}\int_{\mathbb{R}^{n}}\int_{\Lambda_{1}}\bar{\psi}_{0}(y)^{2}dxdy=C_{6}\int_{\mathbb{R}^{n}}\bar{\psi}_{0}(y)^{2}dy.
\end{equation*}

The lemma then follows from \eqref{estimate-2} and
\eqref{estimate-3}.
\end{proof}

The main result is this subsection is as follows.

\begin{thm}
The Lifshitz exponent is bounded from above by $\max\big\{\frac{d_{1}}{2},\frac{\gamma_{1}}{1-\gamma}\big\}+\max\big\{\frac{d_{2}}{2},\frac{\gamma_{2}}{1-\gamma}\big\}$,
i.e.,
\begin{equation*}
\liminf_{E\downarrow E_{0}}\frac{\ln|\ln N(E)|}{\ln(E-E_{0})}\geq-\max\bigg\{\frac{d_{1}}{2},\frac{\gamma_{1}}{1-\gamma}\bigg\}-\max\bigg\{\frac{d_{2}}{2},\frac{\gamma_{2}}{1-\gamma}\bigg\}.
\end{equation*}
\end{thm}
\begin{proof}
Let $\beta_{k}=\max\big\{1,\frac{2}{\alpha(1-\gamma)}\big\}=\frac{2}{d_{k}}\max\big\{\frac{d_{k}}{2},\frac{\gamma_{k}}{1-\gamma}\big\}$, $k=1,2$,
$\Gamma_{L}=\big\{i=(i_{1},i_{2})\in\mathbb{Z}^{d_{1}}\times\mathbb{Z}^{d_{2}}\big||i_{k}|\leq2L^{\beta_{k}},k=1,2\big\}$ and $\Gamma_{L}^{c}=\mathbb{Z}^{d}\backslash\Gamma_{L}$. We define
\begin{equation*}
\begin{split}
W_{\Gamma_{L}}(\omega)&=\frac{1}{L^{d}}\int_{\Lambda_{L}}\sum_{i\in\Gamma_{L}}\frac{\omega_{i}-\omega_{\min}}{(1+|x_{1}-i_{1}|)^{\alpha_{1}}+(1+|x_{2}-i_{2}|)^{\alpha_{2}}}dx,\\
W_{\Gamma_{L}^{c}}(\omega)&=\frac{1}{L^{d}}\int_{\Lambda_{L}}\sum_{i\in\Gamma_{L}^{c}}\frac{\omega_{i}-\omega_{\min}}{(1+|x_{1}-i_{1}|)^{\alpha_{1}}+(1+|x_{2}-i_{2}|)^{\alpha_{2}}}dx
\end{split}
\end{equation*}
so that
$\frac{1}{L^{d}}\int_{\Lambda_{L}}\widehat{W}_{\omega}(x)dx=W_{\Gamma_{L}}(\omega)+W_{\Gamma_{L}^{c}}(\omega)$. Clearly, $W_{\Gamma_{L}}$ and $W_{\Gamma_{L}^{c}}$ are two independent random variables. Moreover, $W_{\Gamma_{L}}(\omega)\leq c_{3}\sum_{i\in\Gamma_{L}}(\omega_{i}-\omega_{\min})$ for some $c_{3}>0$ and there's a constant $c_{4}>0$ such that $\mathbb{P}\{\omega\in\Omega|W_{\Gamma_{L}^{c}}(\omega)\geq c_{4}L^{-2}\}\leq\frac{1}{2}$ for large enough $L$ (see \cite[Lemma 5.2]{KW05}).

By Lemma \ref{lemma-lower-bound} and above analysis, we have for large enough $L$
\begin{equation}\label{estimate-5}
\begin{split}
&\mathbb{P}\big\{\omega\in\Omega\big|E_{0}(H_{\omega,S_{L}}^{D})\leq E\big\}\\
&\quad\quad\geq\mathbb{P}\bigg\{\bigg\{\omega\in\Omega\bigg|E_{0}(H_{\omega,S_{L}}^{D})\leq E\bigg\}\cap\bigg\{\omega\in\Omega\bigg|W_{\Gamma_{L}^{c}}(\omega)<\frac{c_{4}}{L^{2}}\bigg\}\bigg\}\\
&\quad\quad\geq\mathbb{P}\bigg\{\bigg\{\omega\in\Omega\bigg|\sum_{i\in\Gamma_{L}}(\omega_{i}-\omega_{\min})\leq\frac{E-E_{0}}{c_{2}c_{3}}-\frac{c_{1}+c_{2}c_{4}}{c_{2}c_{3}L^{2}}\bigg\}\bigcap\bigg\{\omega\in\Omega\big|W_{\Gamma_{L}^{c}}(\omega)<\frac{c_{4}}{L^{2}}\bigg\}\bigg\}\\
&\quad\quad\geq\frac{1}{2}\mathbb{P}\bigg\{\omega\in\Omega\bigg|\sum_{i\in\Gamma_{L}}(\omega_{i}-\omega_{\min})\leq\frac{E-E_{0}}{c_{2}c_{3}}-\frac{c_{1}+c_{2}c_{4}}{c_{2}c_{3}L^{2}}\bigg\}\\
&\quad\quad=\frac{1}{2}\mathbb{P}\bigg\{\omega\in\Omega\bigg|\sum_{i\in\Gamma_{L}}(\omega_{i}-\omega_{\min})\leq\frac{E-E_{0}}{2c_{2}c_{3}}\bigg\},
\end{split}
\end{equation}
where we set $L=\sqrt{\frac{2(c_{1}+c_{2}c_{4})}{E-E_{0}}}$ for $E$ close enough to $E_{0}$ in the last step. Let $\#\Gamma_{L}$ be the cardinal number of $\Gamma_{L}$. The probability in the last line of \eqref{estimate-5} is bounded from below by
\begin{equation*}
\mathbb{P}\bigg\{\omega\in\Omega\bigg|\omega_{i}-\omega_{\min}\leq\frac{E-E_{0}}{2c_{2}c_{3}\#\Gamma_{L}}\,\,\text{for all}\,\,i\in\Gamma_{L}\bigg\},
\end{equation*}
which, by i.i.d and $\rm(H6)$, is bounded from below by $C\big(\frac{E-E_{0}}{2c_{2}c_{3}\#\Gamma_{L}}\big)^{N\#\Gamma_{L}}$ for $E$ close to $E_{0}$, or equivalently, large enough $L$. Since $\#\Gamma_{L}\leq c_{4}L^{\beta_{1}d_{1}+\beta_{2}d_{2}}$ for some $c_{4}>0$, we have for $E$ close to $E_{0}$, or equivalently, large enough $L$
\begin{equation*}
\begin{split}
\bigg(\frac{E-E_{0}}{2c_{2}c_{3}\#\Gamma_{L}}\bigg)^{N\#\Gamma_{L}}&\geq\bigg(\frac{E-E_{0}}{2c_{2}c_{3}c_{4}L^{\beta_{1}d_{1}+\beta_{2}d_{2}}}\bigg)^{Nc_{4}L^{\beta_{1}d_{1}+\beta_{2}d_{2}}}\\
&=\bigg(c_{5}(E-E_{0})^{1-(\beta_{1}d_{1}+\beta_{2}d_{2})/2}\bigg)^{c_{6}(E-E_{0})^{-(\beta_{1}d_{1}+\beta_{2}d_{2})/2}},
\end{split}
\end{equation*}
where $c_{5},c_{6}>0$. The above estimate and Lemma \ref{sandwich-bound} lead to the theorem.
\end{proof}

%%%%%%%%%%%%%%%%%%%%%%%%%%%%%%%%%%%%%%%%%%%%%%%%%%%%%%%%%%%%%%%%%%%%%%%%%%%%%%%%%%%%%%%%%%%%%%%%%%%%%%%%%%%%%%%%%%%%%%%%%%%%%%%

\subsection{Upper Bound in the Quantum-Classical/Classical-Quantum Regime}

In this section, we study the lower bound of the Lifshitz exponentin the quantum-classical/classical-quantum regime, that is, we assume $\frac{d_{1}}{2}>\frac{\gamma_{1}}{1-\gamma}$ and $\frac{d_{2}}{2}\leq\frac{\gamma_{2}}{1-\gamma}$ (the quantum-classical regime), or
$\frac{d_{1}}{2}\leq\frac{\gamma_{1}}{1-\gamma}$ and $\frac{d_{2}}{2}>\frac{\gamma_{2}}{1-\gamma}$ (the classical-quantum regime). We here focus on the case in the quantum-classical regime.

For $R>0$, we define
\begin{equation*}
\widehat{W}_{\omega,R}(x)=f_{u}\sum_{\substack{i_{1}\in\mathbb{Z}^{d_{1}}\\i_{2}\in\mathbb{Z}^{d_{2}},|i_{2}|>R}}\frac{\min\{\omega_{i}-\omega_{\min},1\}}{(1+|x_{1}-i_{1}|)^{\alpha_{1}}+(1+|x_{2}-i_{2}|)^{\alpha_{2}}},\quad
x\in\mathbb{R}^{d},\,\,\omega\in\Omega
\end{equation*}
and let $\widetilde{W}_{\omega,R}(x,y)=\widehat{W}_{\omega,R}(x)\chi_{G}(y)$ for $x\in\mathbb{R}^{d}$, $y\in\mathbb{R}^{n}$ and $\omega\in\Omega$. By $\rm(H5)$, $0\leq\widetilde{W}_{\omega,R}\leq W_{\omega}$ for all $R>0$ and $\omega\in\Omega$. Let $\widetilde{H}_{\omega,R}=H_{\rm per}+\widetilde{W}_{\omega,R}$ and denote by $\widetilde{H}_{\omega,R}^{\chi,S_{L,1}}$ the restriction of $\widetilde{H}_{\omega,R}$ to $L^{2}(S_{L,1})$ with the Mezincescu boundary condition on $\partial S_{L,1}$.

The main result in the quantum-classical regime is given by

\begin{thm}\label{theorem-upper-bound-q/cl}
Suppose $\frac{d_{1}}{2}>\frac{\gamma_{1}}{1-\gamma}$ and $\frac{d_{2}}{2}\leq\frac{\gamma_{2}}{1-\gamma}$. The Lifshitz exponent in the quantum-classical regime is bounded from below by $\frac{d_{1}}{2}+\frac{\gamma_{2}}{1-\gamma}$. That is,
\begin{equation*}
\limsup_{E\downarrow E_{0}}\frac{\ln|\ln N(E)|}{\ln(E-E_{0})}\leq-\frac{d_{1}}{2}-\frac{\gamma_{2}}{1-\gamma}.
\end{equation*}
\end{thm}

To prove the above theorem, we first find an uniform upper bound on $\widehat{W}_{\omega,R}$ for all $\omega\in\Omega$.

\begin{lem}\label{lemma-tools}
For $R>0$, we define
\begin{equation*}
\widehat{W}_{R}(x)=f_{u}\sum_{\substack{i_{1}\in\mathbb{Z}^{d_{1}}\\i_{2}\in\mathbb{Z}^{d_{2}},|i_{2}|>R}}\frac{1}{(1+|x_{1}-i_{1}|)^{\alpha_{1}}+(1+|x_{2}-i_{2}|)^{\alpha_{2}}},\quad
x\in\mathbb{R}^{d}.
\end{equation*}
There hold the following statements.
\begin{itemize}
\item[\rm(i)] $\widehat{W}_{\omega,R}\leq\widehat{W}_{R}$ pointwise for all $R>0$ and $\omega\in\Omega$.

\item[\rm(ii)] $\widehat{W}_{R}$ is $\mathbb{Z}^{d_{1}}$-periodic in $x_{1}$-direction.

\item[\rm(iii)] There's some constant $c>0$ such that $\sup_{x\in\overline{\Lambda}_{1}}\widehat{W}_{R}(x)\leq\frac{c}{R^{\alpha_{2}(1-\gamma)}}$.
\end{itemize}
\end{lem}
\begin{proof}
$\rm(i)$ and $\rm(ii)$ are trivial. $\rm(iii)$ is a summation version of \cite[Lemma 3.5]{KW05}.
\end{proof}

Next, we estimate a lower bound on the lowest eigenvalue of $\widetilde{H}_{\omega,R}^{S_{L,1},\chi}$.

\begin{lem}\label{lemma-lower-bound-quantum-classical}
Let $R=(rL)^{2/\alpha_{2}(1-\gamma)}$ with $r$ being large enough. Then, the ground state energy, $E_{0}(\widetilde{H}_{\omega,R}^{\chi,S_{L,1}})$, of
$\widetilde{H}_{\omega,R}^{\chi,S_{L,1}}$ satisfies
\begin{equation}\label{estimate-lower-bound-quan-classic}
E_{0}(\widetilde{H}_{\omega,R}^{\chi,S_{L,1}})\geq E_{0}+\frac{C}{L}\int_{\Lambda_{L,1}}\widehat{W}_{\omega,R}(x)dx
\end{equation}
for some $C>0$ and large enough $L$.
\end{lem}
\begin{proof}
We apply Temple's inequality (see \cite[Lemma 6.3]{Ki08}) with variational function $\psi_{L,1}$ (defined in \eqref{ground-states-1}) to the self-adjoint operator $\widetilde{H}_{\omega,R}^{\chi,S_{L,1}}$. We first estimate
\begin{equation*}
\begin{split}
\big\langle\psi_{L,1},\widetilde{H}_{\omega,R}^{\chi,S_{L,1}}\psi_{L,1}\big\rangle-E_{1}(H_{\rm per}^{\chi,S_{L,1}})=\big\langle\psi_{L,1},\widetilde{W}_{\omega,R}\psi_{L,1}\big\rangle+E_{0}(H_{\rm per}^{\chi,S_{L,1}})-E_{1}(H_{\rm per}^{\chi,S_{L,1}})\leq-\frac{C_{1}}{L^{2}}
\end{split}
\end{equation*}
for some $C_{1}>0$, where we have used Lemma \ref{spectral-gap} and the estimate $\big\langle\psi_{L,1},\widetilde{W}_{\omega,R}\psi_{L,1}\big\rangle\leq\frac{c}{R^{\alpha_{2}(1-\gamma)}}$, which follows from Lemma \ref{lemma-tools} and the $\mathbb{Z}^{d}$-periodicity of $\psi_{0}$. Due to the fact that $E_{1}(\widetilde{H}_{\omega,R}^{\chi,S_{L,1}})\geq E_{1}(H_{\rm per}^{\chi,S_{L,1}})$, we have
\begin{equation}\label{inequality-3}
\big\langle\psi_{L,1},\widetilde{H}_{\omega,R}^{\chi,S_{L,1}}\psi_{L,1}\big\rangle-E_{1}(\widetilde{H}_{\omega,R}^{\chi,S_{L,1}})\leq-\frac{C_{1}}{L^{2}}<0.
\end{equation}
It then follows from Temple's inequality that
\begin{equation*}
\begin{split}
E_{0}(\widetilde{H}_{\omega,R}^{\chi,S_{L,1}})&\geq\big\langle\psi_{L,1},\widetilde{H}_{\omega,h}^{\chi,S_{L,1}}\psi_{L,1}\big\rangle-\frac{\big\langle\widetilde{H}_{\omega,R}^{\chi,S_{L,1}}\psi_{L,1},\widetilde{H}_{\omega,R}^{\chi,S_{L,1}}\psi_{L,1}\big\rangle-\big\langle\psi_{L,1},\widetilde{H}_{\omega,R}^{\chi,S_{L,1}}\psi_{L,1}\big\rangle^{2}}{E_{1}(\widetilde{H}_{\omega,R}^{\chi,S_{L,1}})-\big\langle\psi_{L,1},\widetilde{H}_{\omega,R}^{\chi,S_{L,1}}\psi_{L,1}\big\rangle}\\
&\geq E_{0}+\big\langle\psi_{L,1},\widetilde{W}_{\omega,R}\psi_{L,1}\big\rangle-\frac{L^{2}}{C_{1}}\big\langle\widetilde{W}_{\omega,R}\psi_{L,1},\widetilde{W}_{\omega,R}\psi_{L,1}\big\rangle\\
&\geq E_{0}+\frac{1}{2}\big\langle\psi_{L,1},\widetilde{W}_{\omega,R}\psi_{L,1}\big\rangle,
\end{split}
\end{equation*}
where we used \eqref{inequality-3} and the fact
\begin{equation*}
\begin{split}
&\big\langle\widetilde{H}_{\omega,R}^{\chi,S_{L,1}}\psi_{L,1},\widetilde{H}_{\omega,R}^{\chi,S_{L,1}}\psi_{L,1}\big\rangle-\big\langle\psi_{L,1},\widetilde{H}_{\omega,R}^{\chi,S_{L,1}}\psi_{L,1}\big\rangle^{2}\\
&\quad\quad=\big\langle\widetilde{W}_{\omega,R}\psi_{L,1},\widetilde{W}_{\omega,R}\psi_{L,1}\big\rangle-\big\langle\psi_{L,1},\widetilde{W}_{\omega,R}\psi_{L,1}\big\rangle^{2}\\
&\quad\quad\leq\big\langle\widetilde{W}_{\omega,R}\psi_{L,1},\widetilde{W}_{\omega,R}\psi_{L,1}\big\rangle
\end{split}
\end{equation*}
in the second inequality, and used $\big\langle\widetilde{W}_{\omega,R}\psi_{L,1},\widetilde{W}_{\omega,R}\psi_{L,1}\big\rangle\leq\frac{c}{R^{\alpha_{2}(1-\gamma)}}\big\langle\psi_{L,1},\widetilde{W}_{\omega,R}\psi_{L,1}\big\rangle$
and took $r$ be large enough such that $2c\leq C_{1}r^{2}$ in the third inequality.

To finish the proof, we use Lemma \ref{lemma-ground-state} and compute
\begin{equation*}
\begin{split}
\big\langle\psi_{L,1},\widetilde{W}_{\omega,R}\psi_{L,1}\big\rangle=\frac{1}{L}\int_{S_{L,1}}\widehat{W}_{\omega,R}(x)\chi_{G}(y)\psi_{0}(x,y)^{2}dxdy\geq\frac{C_{2}}{L}\int_{\Lambda_{L,1}}\widehat{W}_{\omega,R}(x)dx\int_{\mathbb{R}^{n}}\chi_{G}(y)\bar{\psi}_{0}(y)^{2}dy,
\end{split}
\end{equation*}
which leads to the result.
\end{proof}

Finally, we prove Theorem \ref{theorem-upper-bound-q/cl}.

\begin{proof}[Proof of Theorem \ref{theorem-upper-bound-q/cl}]
By employing \cite[Lemma 4.7, Remark 4.8]{KW05} and choosing $R=(rL)^{2/\alpha_{2}(1-\gamma)}$ with $r$ being large enough so that Lemma \ref{lemma-lower-bound-quantum-classical} holds, there are constants $c_{1},c_{2}>0$ such that
\begin{equation}\label{estimate-lower-bound-quantum-classic-1}
\frac{C}{L}\int_{\Lambda_{L,1}}\widehat{W}_{\omega,R}(x)dx\geq\frac{c_{1}}{(rL)^{2}}\frac{1}{\#\Gamma_{L}}\sum_{i\in\widetilde{\Lambda}}\min\{\omega_{i}-\omega_{\min},1\}-\frac{c_{2}}{L^{\alpha_{1}(1-\gamma)}},
\end{equation}
where $\Gamma_{L}=\big\{i=(i_{1},i_{2})\in\mathbb{Z}^{2}\big||i_{1}|\leq\frac{L}{8},\,\,R<|i_{2}|\leq2R\big\}$ and $\#\Gamma_{L}$ is the cardinal number of $\Gamma_{L}$.

Using \eqref{estimate-lower-bound-quan-classic}, \eqref{estimate-lower-bound-quantum-classic-1} and the fact $E_{0}(H_{\omega}^{\chi,S_{L,1}})\geq
E_{0}(\widetilde{H}_{\omega,R}^{\chi,S_{L,1}})$, we obtain
\begin{equation}\label{large-dev-quan-cla}
\begin{split}
\mathbb{P}\Big\{\omega\in\Omega\Big|E_{0}(H_{\omega}^{S_{L,1},\chi})\leq E\Big\}\leq\mathbb{P}\bigg\{\omega\in\Omega\bigg|\frac{1}{\#\Gamma_{L}}\sum_{i\in\Gamma_{L}}\xi_{i}(\omega)\leq\frac{(rL)^{2}}{c_{1}}\bigg(E-E_{0}+\frac{c_{2}}{L^{\alpha_{1}(1-\gamma)}}\bigg)\bigg\},
\end{split}
\end{equation}
where $\xi_{i}(\omega)=\min\{\omega_{i}-\omega_{\min},1\}$ for $i\in\Gamma_{L}$. Since $\alpha_{1}(1-\gamma)>2$ by $\frac{d_{1}}{2}>\frac{\gamma_{1}}{1-\gamma}$, $\frac{(rL)^{2}}{c_{1}}\frac{c_{2}}{L^{\alpha_{1}(1-\gamma)}}\leq\frac{1}{r}$ for $L$ large enough. Setting $L=\sqrt{\frac{c_{1}}{r^{3}(E-E_{0})}}$, we have $\frac{(rL)^{2}}{c_{1}}(E-E_{0})=\frac{1}{r}$. Therefore, the probability on the right-hand side of \eqref{large-dev-quan-cla} is bounded from above by $\mathbb{P}\big\{\omega\in\Omega\big|\frac{1}{\#\Gamma_{L}}\sum_{i\in\Gamma_{L}}\xi_{i}(\omega)\leq\frac{2}{r}\big\}$, which is the probability of a large deviation event (see \cite{DZ10}) if $r$ is large enough and $r>\frac{2}{\mathbb{E}(\xi)}$, where $\xi$ is the general representation of the i.i.d random variables $\xi_{i},i\in\Gamma_{L}$. Hence, we can argue as in the proof of \cite[Lemma 6.4]{Ki08} that there's some $c_{3}>0$ such that
\begin{equation}\label{large-dev-estimate}
\mathbb{P}\Big\{\omega\in\Omega\Big|E_{0}(H_{\omega}^{S_{L,1},\chi})\leq E\Big\}\leq e^{-c_{3}\#\Gamma_{L}}.
\end{equation}
The theorem then follows from \eqref{large-dev-estimate}, the fact that
\begin{equation*}
\#\Gamma_{L}\geq c_{4}L^{d_{1}}(rL)^{2\gamma_{2}/(1-\gamma)}=c_{4}c_{1}^{\frac{d_{1}}{2}+\frac{\gamma_{2}}{1-\gamma}}r^{-\frac{3d_{1}}{2}-\frac{\gamma_{2}}{1-\gamma}}(E-E_{0})^{-\frac{d_{1}}{2}-\frac{\gamma_{2}}{1-\gamma}}
\end{equation*}
for some $c_{4}>0$ and Lemma \ref{sandwich-bound}.
\end{proof}

The result in the classical-quantum regime can be proven analogously. We state the result without proof.

\begin{thm}\label{theorem-upper-bound-cl/q}
Suppose $\frac{d_{1}}{2}\leq\frac{\gamma_{1}}{1-\gamma}$ and $\frac{d_{2}}{2}>\frac{\gamma_{2}}{1-\gamma}$. The Lifshitz exponent in the classical-quantum regime is bounded from below by $\frac{\gamma_{1}}{1-\gamma}+\frac{d_{2}}{2}$. That is,
\begin{equation*}
\limsup_{E\downarrow E_{0}}\frac{\ln|\ln N(E)|}{\ln(E-E_{0})}\leq-\frac{\gamma_{1}}{1-\gamma}-\frac{d_{2}}{2}.
\end{equation*}
\end{thm}

\subsection{Upper Bound in the Quantum Regime}

In this section, we study the lower bound of the Lifshitz exponent in the quantum regime, that is, we assume $\frac{d_{k}}{2}>\frac{\gamma_{k}}{1-\gamma}$, $k=1,2$.

For any $h>0$, we define
\begin{equation*}
\widehat{W}_{\omega,h}(x)=f_{u}\sum_{i\in\mathbb{Z}^{d}}\min\{\omega_{i}-\omega_{\min},h\}\chi_{F}(x-i),\quad x\in\mathbb{R}^{d},\,\,\omega\in\Omega
\end{equation*}
and set $\widetilde{W}_{\omega,h}(x,y)=\widehat{W}_{\omega,h}(x)\chi_{G}(y)$ for $x\in\mathbb{R}^{d}$, $y\in\mathbb{R}^{n}$ and $\omega\in\Omega$. By assumption $\rm(H2)(ii)$, we can find a constant $f_{u}>0$ and two Borel sets $F\subset\Lambda_{1}$ and $G\subset\mathbb{R}^{n}$ such that $f(x,y)\geq
f_{u}\chi_{F}(x)\chi_{G}(y)$ for all $x\in\mathbb{R}^{d}$ and $y\in\mathbb{R}^{n}$. It then follows that $0\leq\widetilde{W}_{\omega,h}\leq\min\{f_{u}h,W_{\omega}\}$ for any $h>0$ and $\omega\in\Omega$. Let $\widetilde{H}_{\omega,h}=H_{\rm per}+\widetilde{W}_{\omega,h}$ and $\widetilde{H}_{\omega,h}^{\chi,S_{L}}$ be the restriction of $\widetilde{H}_{\omega,h}$ to $L^{2}(S_{L})$ with the Mezincescu boundary condition on $\partial S_{L}$.

Analogous to Lemma \ref{lemma-lower-bound-quantum-classical}, we estimate a lower bound on the lowest eigenvalue of $\widetilde{H}_{\omega,h}^{\chi,S_{L}}$.

\begin{lem}\label{lemma-lower-bound-quantum}
Let $h=\frac{C_{per}}{3f_{u}L^{2}}$. Then, the ground state energy, $E_{0}(\widetilde{H}_{\omega,h}^{\chi,S_{L}})$, of $\widetilde{H}_{\omega,h}^{\chi,S_{L}}$ satisfies
\begin{equation*}
E_{0}(\widetilde{H}_{\omega,h}^{\chi,S_{L}})\geq E_{0}+\frac{C}{L^{d}}\int_{\Lambda_{L}}\widehat{W}_{\omega,h}(x)dx
\end{equation*}
for some $C>0$ and large enough $L$.
\end{lem}
\begin{proof}
Since $E_{1}(\widetilde{H}_{\omega,h}^{\chi,S_{L}})\geq E_{1}(H_{\rm per}^{\chi,S_{L}})$ and $\langle\psi_{L},\widetilde{H}_{\omega,h}^{\chi,S_{L}}\psi_{L}\rangle-E_{1}(H_{\rm per}^{\chi,S_{L}})\leq-\frac{2C_{\rm per}}{3L^{2}}$ by Lemma
\ref{spectral-gap-1}, we have $\big\langle\psi_{L},\widetilde{H}_{\omega,h}^{\chi,S_{L}}\psi_{L}\big\rangle-E_{1}(\widetilde{H}_{\omega,h}^{\chi,S_{L}})\leq-\frac{2C_{per}}{3L^{2}}<0$.
It then follows from Temple's inequality with variational function $\psi_{L}$ that $E_{0}(\widetilde{H}_{\omega,h}^{\chi,S_{L}})\geq E_{0}+\frac{1}{2}\big\langle\psi_{L},\widetilde{W}_{\omega,h}\psi_{L}\big\rangle$, which leads to the result. We refer to Lemma \ref{lemma-lower-bound-quantum-classical} for detailed arguments.
\end{proof}

We proceed to the main result in this subsection.

\begin{thm}
The Lifshitz exponent in the quantum regime is bounded from below by $\frac{d}{2}$, i.e.,
\begin{equation*}
\limsup_{E\downarrow E_{0}}\frac{\ln|\ln N(E)|}{\ln(E-E_{0})}\leq-\frac{d}{2}.
\end{equation*}
\end{thm}
\begin{proof}
Let $h=\frac{C_{per}}{3f_{u}L^{2}}$ so that Lemma \ref{lemma-lower-bound-quantum} holds. We first claim that \begin{equation}\label{estimate-lower-bound-quantum-1}
E_{0}(\widetilde{H}_{\omega,h}^{S_{L},\chi})\geq E_{0}+\frac{Cf_{u}|F|h}{L^{d}}\#\Big\{i\in\mathbb{Z}^{d}\cap\Lambda_{L}\Big|\omega_{i}-\omega_{\min}\geq h\Big\}
\end{equation}
for $L\in\mathbb{N}$, where $|F|$ is the $d$-dimensional Lebesgue measure of $F$ and $\#\{\cdot\}$ is the cardinal number of the set $\{\cdot\}$. Indeed, we calculate
\begin{equation*}
\begin{split}
\int_{\Lambda_{L}}\widehat{W}_{\omega,h}(x)dx\geq f_{u}\int_{\Lambda_{L}}\sum_{i\in\mathbb{Z}^{d}\cap\Lambda_{L}}\min\{\omega_{i}-\omega_{\min},h\}\chi_{F}(x-i)dx=f_{u}|F|\sum_{i\in\mathbb{Z}^{d}\cap\Lambda_{L}}\min\{\omega_{i}-\omega_{\min},h\},
\end{split}
\end{equation*}
where we have used the fact that $\int_{\Lambda_{L}}\chi_{F}(x-i)dx=|F|$ for all $i\in\mathbb{Z}^{d}\cap\Lambda_{L}$. \eqref{estimate-lower-bound-quantum-1} then follows from
\begin{equation*}
\sum_{i\in\mathbb{Z}^{d}\cap\Lambda_{L}}\min\{\omega_{i}-\omega_{\min},h\}\geq h\#\Big\{i\in\mathbb{Z}^{d}\cap\Lambda_{L}\Big|\omega_{i}-\omega_{\min}\geq
h\Big\}.
\end{equation*}

Considering \eqref{estimate-lower-bound-quantum-1} and the fact $H_{\omega}^{\chi,S_{L}}\geq\widetilde{H}_{\omega,h}^{\chi,S_{L}}$, hence $E_{0}(H_{\omega}^{S_{L},\chi})\geq E_{0}(\widetilde{H}_{\omega,h}^{S_{L},\chi})$, we obtain
\begin{equation*}
\begin{split}
\mathbb{P}\Big\{\omega\in\Omega\Big|E_{0}(H_{\omega}^{\chi,S_{L}})\leq E\Big\}&\leq\mathbb{P}\bigg\{\omega\in\Omega\bigg|\frac{1}{L^{d}}\#\Big\{i\in\mathbb{Z}^{d}\cap\Lambda_{L}\Big|\omega_{i}-\omega_{\min}\geq
h\Big\}\leq\frac{E-E_{0}}{Cf_{u}|F|h}\bigg\}\\
&=\mathbb{P}\bigg\{\omega\in\Omega\bigg|\frac{1}{L^{d}}\#\Big\{i\in\mathbb{Z}^{d}\cap\Lambda_{L}\Big|\omega_{i}-\omega_{\min}< h\Big\}>\frac{E-E_{0}}{Cf_{u}|F|h}\bigg\}.
\end{split}
\end{equation*}
Let $\xi_{i}$ be the characteristic function of the set $\big\{\omega\in\Omega\big|\omega_{i}-\omega_{\min}< h\big\}$ for $i\in\mathbb{Z}^{d}\cap\Lambda_{L}$. It's easy to see that $\{\xi_{i}\}_{i\in\mathbb{Z}^{d}}$ are nonnegative i.i.d random variables with expectation $\mathbb{E}(\xi)\in(0,1)$ for small $h$, since we have assumed that $\text{supp}\mathbb{P}_{0}$ contains at least two points, where $\xi$ is the general representation of $\{\xi_{i}\}_{i\in\mathbb{Z}^{d}}$. Pick any $r\in(\mathbb{E}(\xi),1)$ and set $L=\sqrt{\frac{C|F|C_{\rm per}r}{3}}(E-E_{0})^{-\frac{1}{2}}$ for $E>E_{0}$. We have $h=\frac{E-E_{0}}{Cf_{u}|F|r}$ and
\begin{equation}\label{large-deviation-quan}
\mathbb{P}\Big\{\omega\in\Omega\Big|E_{0}(H_{\omega}^{\chi,S_{L}})\leq E\Big\}\leq\mathbb{P}\bigg\{\omega\in\Omega\bigg|\frac{1}{L^{d}}\sum_{i\in\mathbb{Z}^{d}\cap\Lambda_{L}}\xi_{i}(\omega)>r\bigg\},
\end{equation}
which is the probability of a large deviation event. By picking $E$ close to $E_{0}$ so that $L$ is large and $h$ is small, there's some constant $C_{1}>0$ so that the probability in the right-hand side of \eqref{large-deviation-quan} is bounded from above by
\begin{equation*}
e^{-C_{1}L^{d}}=e^{-3^{-d/2}C_{1}(C|F|C_{\rm per}r)^{d/2}(E-E_{0})^{-d/2}}.
\end{equation*}
Considering Lemma \ref{sandwich-bound}, we obtain the result.
\end{proof}

%%%%%%%%%%%%%%%%%%%%%%%%%%%%%%%%%%%%%%%%%%%%%%%%%%%%%%%%%%%%%%%%%%%%%%%%%%%%%%%%%%%%%%%%%%%%%%%%%%%%%%%%%%%%%%%%%%%%%%%%%%%%%%%%

\subsection{Upper Bound in the Classical Regime}

In this section, we study the lower bound of the Lifshitz exponent in the classical regime, that is, we assume $\frac{d_{k}}{2}\leq\frac{\gamma_{k}}{1-\gamma}$, $k=1,2$.

Let
$\beta_{k}=\frac{2}{d_{k}}\frac{\gamma_{k}}{1-\gamma}=\frac{2}{\alpha_{k}(1-\gamma)}$, $k=1,2$. We define
\begin{equation*}
\widehat{W}_{\omega,L}(x)=f_{u}\sum_{\substack{i_{1}\in\mathbb{Z}^{d_{1}},|i_{1}|>L^{\beta_{1}}\\i_{2}\in\mathbb{Z}^{d_{2}},|i_{2}|>L^{\beta_{2}}}}\frac{\min\{\omega_{i}-\omega_{\min},1\}}{|x_{1}-i_{1}|^{\alpha_{1}}+|x_{2}-i_{2}|^{\alpha_{2}}},\quad
x\in\mathbb{R}^{d},\,\,\omega\in\Omega
\end{equation*}
and set $\widetilde{W}_{\omega,L}(x,y)=\widehat{W}_{\omega,L}(x)\chi_{G}(y)$ for $x\in\mathbb{R}^{d}$, $y\in\mathbb{R}^{n}$ and $\omega\in\Omega$. By assumption $\rm(H5)$, $0\leq\widetilde{W}_{\omega,L}\leq W_{\omega}$ for all $L\geq1$ and $\omega\in\Omega$. Let $\widetilde{H}_{\omega,L}=H_{\rm
per}+\widetilde{W}_{\omega,L}$ and $\widetilde{H}_{\omega,L}^{\chi,S_{L}}$ be the restriction of $\widetilde{H}_{\omega,L}$ to $L^{2}(S_{L})$ with the Mezincescu boundary condition on $\partial S_{L}$.

We estimate an upper bound on the lowest eigenvalue of $\widetilde{H}_{\omega,L}^{\chi,S_{1}}$.

\begin{lem}\label{lemma-lower-bound-classical}
The ground state energy, $E_{0}(\widetilde{H}_{\omega,L}^{\chi,S_{1}})$, of $\widetilde{H}_{\omega,L}^{\chi,S_{1}}$ satisfies
\begin{equation*}
E_{0}(\widetilde{H}_{\omega,L}^{\chi,S_{1}})\geq E_{0}+C\int_{\Lambda_{1}}\widehat{W}_{\omega,L}(x)dx
\end{equation*}
for some $C>0$ and large enough $L$.
\end{lem}
\begin{proof}
It's not difficult to see that $\sup_{x\in\overline{\Lambda}_{1},y\in\mathbb{R}^{n}}\widetilde{W}_{\omega,L}(x,y)\leq\frac{C_{per}}{3}$ for large enough $L$, which together with Lemma \ref{spectral-gap-1}, implies that $\big\langle\psi_{1},\widetilde{H}_{\omega,L}^{\chi,S_{1}}\psi_{1}\big\rangle-E_{1}(\widetilde{H}_{\omega,L}^{\chi,S_{1}})\leq-\frac{2C_{per}}{3}<0$,
where $\psi_{1}$ is the ground state of $H_{{\rm per},S_{1}}^{\chi}$ defined in \eqref{ground-states}. Applying Temple's inequality with variational function $\psi_{1}$ to the self-adjoint operator $\widetilde{H}_{\omega,L}^{\chi,S_{1}}$, we obtain $E_{0}(\widetilde{H}_{\omega,L}^{\chi,S_{1}})\geq
E_{0}+\frac{1}{2}\big\langle\psi_{1},\widetilde{W}_{\omega,L}\psi_{1}\big\rangle$. The lemma then follows. We refer to Lemma
\ref{lemma-lower-bound-quantum-classical} for detailed arguments.
\end{proof}

The main result in this subsection is stated as follows.

\begin{thm}
The Lifshitz exponent in the classical regime is bounded from below by $\frac{\gamma}{1-\gamma}$, i.e.,
\begin{equation*}
\limsup_{E\downarrow E_{0}}\frac{\ln|\ln N(E)|}{\ln(E-E_{0})}\leq-\frac{\gamma}{1-\gamma}.
\end{equation*}
\end{thm}
\begin{proof}
We first claim that there is a constant $C>0$ (independent of $\omega$ and $L$) such that
\begin{equation}\label{estimate-lower-bound-classic-1}
\int_{\Lambda_{1}}\widehat{W}_{\omega,L}(x)dx\geq\frac{C}{L^{2}}\frac{1}{\#\Gamma_{L}}\sum_{i\in\Gamma_{L}}\min\{\omega_{i}-\omega_{\min},1\}
\end{equation}
for large enough $L$, where $\Gamma_{L}=\big\{i=(i_{1},i_{2})\in\mathbb{Z}^{d_{1}}\times\mathbb{Z}^{d_{2}}\big|2L^{\beta_{k}}<|i_{k}|\leq4L^{\beta_{k}},\,\,k=1,2\big\}$
and $\#\Gamma_{L}$ is the cardinal number of $\Gamma_{L}$. Indeed, by neglecting a positive term, we have
\begin{equation*}
\int_{\Lambda_{1}}\widehat{W}_{\omega,L}(x)dx\geq
f_{u}\sum_{i\in\Gamma_{L}}\min\{\omega_{i}-\omega_{\min},1\}\int_{\Lambda_{1}}\frac{1}{|x_{1}-i_{1}|^{\alpha_{1}}+|x_{2}-i_{2}|^{\alpha_{2}}}dx.
\end{equation*}
For the integral on the right-hand side, we have
\begin{equation*}
\int_{\Lambda_{1}}\frac{1}{|x_{1}-i_{1}|^{\alpha_{1}}+|x_{2}-i_{2}|^{\alpha_{2}}}dx\geq\frac{1}{5^{\alpha_{1}}L^{\alpha_{1}\beta_{1}}+5^{\alpha_{2}}L^{\alpha_{2}\beta_{2}}}
\end{equation*}
for all $i\in\Gamma_{L}$. Since $\beta_{k}=\frac{2}{d_{k}}\frac{\gamma_{k}}{1-\gamma}=\frac{2}{\alpha_{k}(1-\gamma)}$, $k=1,2$ and $\#\Gamma_{L}\geq
C_{u}L^{\beta_{1}d_{1}}L^{\beta_{2}d_{2}}=C_{u}L^{2\gamma/(1-\gamma)}$ for some $C_{u}>0$, we obtain
\eqref{estimate-lower-bound-classic-1}.

Using Lemma \ref{lemma-lower-bound-classical}, \eqref{estimate-lower-bound-classic-1} and the fact $E_{0}(H_{\omega}^{\chi,S_{1}})\geq E_{0}(\widetilde{H}_{\omega,L}^{\chi,S_{1}})$, we obtain
\begin{equation}\label{inequality-2}
\begin{split}
\mathbb{P}\Big\{\omega\in\Omega\Big|E_{0}(H_{\omega}^{S_{1},\chi})\leq
E\Big\}\leq\mathbb{P}\bigg\{\omega\in\Omega\bigg|\frac{1}{\#\Gamma_{L}}\sum_{i\in\Gamma_{L}}\xi_{i}(\omega)\leq\frac{L^{2}(E-E_{0})}{C}\bigg\},
\end{split}
\end{equation}
where $\xi_{i}(\omega)=\min\{\omega_{i}-\omega_{\min},1\}$ for $i\in\Gamma_{L}$. Obviously, $\{\xi_{i}\}_{i\in\Gamma_{L}}$ are i.i.d random variables with expectation $\mathbb{E}(\xi)\in(0,1)$, where $\xi$ is the general representation of $\{\xi_{i}\}_{i\in\Gamma_{L}}$. Fix any $r\in(0,\mathbb{E}(\xi))$ and let $L=\sqrt{Cr}(E-E_{0})^{-\frac{1}{2}}$ for $E>E_{0}$. Hence, whenever $E$ is close to $E_{0}$, $L$ is large. Large deviation argument applied to the probability on the right-hand side of \eqref{inequality-2} leads to
\begin{equation*}
\mathbb{P}\big\{\omega\in\Omega\big|E_{0}(H_{\omega}^{\chi,S_{1}})\leq E\big\}\leq e^{-C_{2}\#\Gamma_{L}}\leq e^{-C_{2}C_{u}(Cr)^{\frac{\gamma}{1-\gamma}}(E-E_{0})^{-\frac{\gamma}{1-\gamma}}}
\end{equation*}
for some $C_{2}>0$, which together with Lemma \ref{sandwich-bound}, gives the result.
\end{proof}

%%%%%%%%%%%%%%%%%%%%%%%%%%%%%%%%%%%%%%%%%%%%%%%%%%%%%%%%%%%%%%%%%%%%%%%%%%%%%%%%%%%%%%%%%%%%%%%%%%%%%%%%%%%%%%%%%%%%%%%%%%%%%%%

\appendix

\section{Neumann Laplacian on Cuboids}\label{sec-neumann-laplacian}

Let $\Lambda=\prod_{i=1}^{d}(a_{i},b_{i})$ be an open cuboid in $\mathbb{R}^{d}$ and $\partial\Lambda$ be its boundary. Let
\begin{equation*}
\Gamma_{i}(a_{i})=\partial\Lambda\cap\{x_{i}=a_{i}\},\quad\Gamma_{i}(b_{i})=\partial\Lambda\cap\{x_{i}=b_{i}\},\quad i=1,\dots,d
\end{equation*}
be the surfaces of $\Lambda$. We consider the following homogeneous Neumann problem of $-\Delta$ on $\overline{\Lambda}$:
\begin{equation}\label{neumann-problem}
\left\{ \begin{aligned}
&-\Delta\phi=E\phi\quad\text{in}\quad\Lambda\\
&\partial_{x_{i}}\phi|_{\Gamma_{i}(a_{i})}=0=\partial_{x_{i}}\phi|_{\Gamma_{i}(b_{i})},\,\,i=1,\dots,d.
\end{aligned} \right.
\end{equation}
Solving the problem \eqref{neumann-problem} is often rephrased as to find eigenvalues and corresponding eigenfunctions of the Neumann Laplacian $-\Delta^{N}_{\Lambda}$. Solutions to \eqref{neumann-problem} (or eigenvalues and eigenfunctions of $-\Delta^{N}_{\Lambda}$) are given by
\begin{equation*}
\begin{split}
E_{M}&=\pi^{2}\sum_{i=1}^{d}\frac{M_{i}^{2}}{(b_{i}-a_{i})^{2}},\\
\phi_{M}(x)&=2^{d/2}\bigg(\prod_{i=1}^{d}(b_{i}-a_{i})\bigg)^{-1/2}\prod_{i=1}^{d}\cos\bigg(\frac{M_{i}\pi(x_{i}-a_{i})}{b_{i}-a_{i}}\bigg),\quad x=(x_{1},\dots,x_{d})\in\Lambda
\end{split}
\end{equation*}
with $\int_{\Lambda}|\phi_{M}(x)|^{2}dx=1$ for $M=(M_{1},\dots,M_{d})\in\mathbb{N}_{0}^{d}=(\mathbb{N}\cup\{0\})^{d}$.

\begin{lem}\label{lemma-app-heat-kernel-estimate-1}
The integral kernel $e^{t\Delta_{\Lambda}^{N}}(\cdot,\cdot)$ of $e^{t\Delta_{\Lambda}^{N}}$ satisfies
\begin{equation*}
\int_{\Lambda}\big|e^{t\Delta_{\Lambda}^{N}}(x,\bar{x})\big|^{2}d\bar{x}\leq\prod_{i=1}^{d}\bigg(\frac{2}{b_{i}-a_{i}}+\frac{1}{\sqrt{2\pi t}}\bigg),\quad x\in\Lambda
\end{equation*}
for all $t>0$.
\end{lem}
\begin{proof}
Since $\big\{\phi_{M};M\in\mathbb{N}_{0}^{d}\big\}$ forms a orthonormal basis of $L^{2}(\Lambda)$, any $f\in L^{2}(\Lambda)$ has the expansion $f=\sum_{M\in\mathbb{N}_{0}^{d}}\big\langle\phi_{M},f\big\rangle_{L^{2}(\Lambda)}\phi_{M}$. It then follows that for any $x\in\Lambda$
\begin{equation*}
\begin{split}
\big(e^{t\Delta_{\Lambda}^{N}}f\big)(x)=\sum_{M\in\mathbb{N}_{0}^{d}}\big\langle\phi_{M},f\big\rangle_{L^{2}(\Lambda)}e^{-E_{M}t}\phi_{M}(x)=\int_{\Lambda}\bigg(\sum_{M\in\mathbb{N}_{0}^{d}}e^{-E_{M}t}\phi_{M}(x)\phi_{M}(\bar{x})\bigg)f(\bar{x})d\bar{x},
\end{split}
\end{equation*}
which implies that $e^{t\Delta_{\Lambda}^{N}}(x,\bar{x})=\sum_{M\in\mathbb{N}_{0}^{d}}e^{-E_{M}t}\phi_{M}(x)\phi_{M}(\bar{x})$, $x,\bar{x}\in\Lambda$. We then compute for any $x\in\Lambda$
\begin{equation}\label{equality-app-1}
\begin{split}
\int_{\Lambda}\big|e^{t\Delta_{\Lambda}^{N}}(x,\bar{x})\big|^{2}d\bar{x}&=\int_{\Lambda}\bigg(\sum_{M,N\in\mathbb{N}_{0}^{d}}e^{-(E_{M}+E_{N})t}\phi_{M}(x)\phi_{M}(\bar{x})\phi_{N}(x)\phi_{N}(\bar{x})\bigg)d\bar{x}\\
&=\sum_{M,N\in\mathbb{N}_{0}^{d}}e^{-(E_{M}+E_{N})t}\phi_{M}(x)\phi_{N}(x)\int_{\Lambda}\phi_{M}(\bar{x})\phi_{N}(\bar{x})d\bar{x}\\
&=\sum_{M\in\mathbb{N}_{0}^{d}}e^{-2E_{M}t}|\phi_{M}(x)|^{2},
\end{split}
\end{equation}
where we have used the fact that the integral $\int_{\Lambda}\phi_{M}(\bar{x})\phi_{N}(\bar{x})d\bar{x}=\langle\phi_{M},\phi_{N}\rangle_{L^{2}(\Lambda)}$
equals $1$ if $M=N$ and equals $0$ if $M\neq N$.

We next estimate the last term in \eqref{equality-app-1}. Since $\phi_{M}(x)\leq2^{d/2}\Big(\prod_{i=1}^{d}(b_{i}-a_{i})\Big)^{-1/2}$ for all $x\in\Lambda$ and all $M\in\mathbb{N}_{0}^{d}$, we obtain for any $x\in\Lambda$
\begin{equation}\label{estimate-7}
\begin{split}
\sum_{M\in\mathbb{N}_{0}^{d}}e^{-2E_{M}t}|\phi_{M}(x)|^{2}&\leq2^{d}\bigg(\prod_{i=1}^{d}(b_{i}-a_{i})\bigg)^{-1}\sum_{M\in\mathbb{N}_{0}^{d}}e^{-2E_{M}t}\\
&=2^{d}\bigg(\prod_{i=1}^{d}(b_{i}-a_{i})\bigg)^{-1}\prod_{i=1}^{d}\bigg(\sum_{M_{i}\in\mathbb{N}_{0}}e^{-\frac{2\pi^{2}M_{i}^{2}}{(b_{i}-a_{i})^{2}}t}\bigg).
\end{split}
\end{equation}
The sums in the last step of above estimates can be estimated by using Gaussian integrals. More precisely, for any $i\in\{1,\dots,d\}$, we have
\begin{equation*}
\begin{split}
\sum_{M_{i}\in\mathbb{N}_{0}}e^{-\frac{2\pi^{2}M_{i}^{2}}{(b_{i}-a_{i})^{2}}t}\leq1+\int_{0}^{\infty}e^{-\frac{2\pi^{2}u^{2}}{(b_{i}-a_{i})^{2}}t}du=1+\frac{b_{i}-a_{i}}{\sqrt{8\pi
t}},
\end{split}
\end{equation*}
which together with \eqref{equality-app-1} and \eqref{estimate-7} leads to the result of the lemma.
\end{proof}

\section{Laplacian on Cuboids with Mixed Boundary Conditoins}

Let $\Lambda_{d}=\prod_{i=1}^{d}(a_{i},b_{i})$ be an open cuboid in $\mathbb{R}^{d}$ and $\partial\Lambda_{d}$ be its boundary. Let $\Lambda_{n}=\prod_{j=1}^{n}(c_{j},d_{j})$ be an open cuboid in $\mathbb{R}^{n}$ and $\partial\Lambda_{n}$ be its boundary. Set $S=\Lambda_{d}\times\Lambda_{n}$. Denote by $-\Delta_{S}^{X,D}$ the negative Laplacian $-\Delta$ restricted to $L^{2}(S)$ with $X$
boundary conditions on $\partial\Lambda_{d}\times\Lambda_{n}$ and Dirichlet boundary conditions on $\Lambda_{d}\times\partial\Lambda_{n}$, where $X=D$ and $X=N$ refer to Dirichlet and Neumann boundary conditions, respectively.

The eigenvalues of $-\Delta_{S}^{N,D}$ are given by
\begin{equation}\label{eigenvalue-N-D}
E_{M,N}^{N,D}=\pi^{2}\bigg(\sum_{i=1}^{d}\frac{M_{i}^{2}}{(b_{i}-a_{i})^{2}}+\sum_{j=1}^{n}\frac{N_{j}^{2}}{(d_{j}-c_{j})^{2}}\bigg),\quad M\in\mathbb{N}_{0}^{d},\,\,N\in\mathbb{N}^{n},
\end{equation}
where we use the same notation for the Neumann boundary conditions and the multiple index, and it should not cause any confusion. Similarly, the eigenvalues of $-\Delta_{S}^{D,D}$ are given by
\begin{equation}\label{eigenvalue-D-D}
E_{M,N}^{D,D}=\pi^{2}\bigg(\sum_{i=1}^{d}\frac{M_{i}^{2}}{(b_{i}-a_{i})^{2}}+\sum_{j=1}^{n}\frac{N_{j}^{2}}{(d_{j}-c_{j})^{2}}\bigg), \quad M\in\mathbb{N}^{d},\,\,N\in\mathbb{N}^{n}.
\end{equation}

\begin{lem}\label{lemma-app-difference-of-trace}
The difference between ${\rm Tr}\big[e^{\Delta_{S}^{N,D}}\big]$ and ${\rm Tr}\big[e^{\Delta_{S}^{D,D}}\big]$ satisfies
\begin{equation*}
0\leq{\rm Tr}\Big[e^{\Delta_{S}^{N,D}}\Big]-{\rm Tr}\Big[e^{\Delta_{S}^{D,D}}\Big]\leq\bigg[\prod_{i=1}^{d}\bigg(1+\frac{b_{i}-a_{i}}{\sqrt{4\pi}}\bigg)-\prod_{i=1}^{d}\bigg(\frac{b_{i}-a_{i}}{\sqrt{4\pi}}-1\bigg)\bigg]\bigg[\prod_{j=1}^{n}\frac{d_{j}-c_{j}}{\sqrt{4\pi}}\bigg].
\end{equation*}
\end{lem}
\begin{proof}
Using \eqref{eigenvalue-N-D}, ${\rm Tr}\big[e^{\Delta_{S}^{N,D}}\big]$ can be written as \begin{equation}\label{app-equality-1}
\begin{split}
{\rm Tr}\Big[e^{\Delta_{S}^{N,D}}\Big]&=\sum_{M\in\mathbb{N}_{0}^{d}}\sum_{N\in\mathbb{N}^{n}}e^{-E_{M,N}^{N,D}}\\
%&=\sum_{M\in\mathbb{N}_{0}^{d}}\sum_{N\in\mathbb{N}^{n}}{\rm %exp}\bigg\{-\pi^{2}\sum_{i=1}^{d}\frac{M_{i}^{2}}{(b_{i}-a_{i})^{2}}\bigg\}{\rm
%exp}\bigg\{-\pi^{2}\sum_{j=1}^{n}\frac{N_{j}^{2}}{(d_{j}-c_{j})^{2}}\bigg\}\\
&=\sum_{M\in\mathbb{N}_{0}^{d}}{\rm exp}\bigg\{-\pi^{2}\sum_{i=1}^{d}\frac{M_{i}^{2}}{(b_{i}-a_{i})^{2}}\bigg\}\times\sum_{N\in\mathbb{N}^{n}}{\rm
exp}\bigg\{-\pi^{2}\sum_{j=1}^{n}\frac{N_{j}^{2}}{(d_{j}-c_{j})^{2}}\bigg\}.
\end{split}
\end{equation}
Similarly, \eqref{eigenvalue-D-D} gives
\begin{equation}\label{app-equality-2}
{\rm Tr}\Big[e^{\Delta_{S}^{D,D}}\Big]=\sum_{M\in\mathbb{N}^{d}}{\rm exp}\bigg\{-\pi^{2}\sum_{i=1}^{d}\frac{M_{i}^{2}}{(b_{i}-a_{i})^{2}}\bigg\}\times\sum_{N\in\mathbb{N}^{n}}{\rm exp}\bigg\{-\pi^{2}\sum_{j=1}^{n}\frac{N_{j}^{2}}{(d_{j}-c_{j})^{2}}\bigg\}.
\end{equation}

For the first summation in the last line of \eqref{app-equality-1}, we have
\begin{equation}\label{app-inequality-1}
\begin{split}
&\sum_{M\in\mathbb{N}_{0}^{d}}{\rm exp}\bigg\{-\pi^{2}\sum_{i=1}^{d}\frac{M_{i}^{2}}{(b_{i}-a_{i})^{2}}\bigg\}\\
&\quad\quad=\prod_{i=1}^{d}\bigg(\sum_{M_{i}\in\mathbb{N}_{0}}e^{-\frac{\pi^{2}M_{i}^{2}}{(b_{i}-a_{i})^{2}}}\bigg)\leq\prod_{i=1}^{d}\bigg(1+\int_{0}^{\infty}e^{-\frac{\pi^{2}u^{2}}{(b_{i}-a_{i})^{2}}}du\bigg)=\prod_{i=1}^{d}\bigg(1+\frac{b_{i}-a_{i}}{\sqrt{4\pi}}\bigg).
\end{split}
\end{equation}
For the first summation on the right hand side of \eqref{app-equality-2}, we estimate
\begin{equation}\label{app-inequality-2}
\begin{split}
&\sum_{M\in\mathbb{N}^{d}}{\rm exp}\bigg\{-\pi^{2}\sum_{i=1}^{d}\frac{M_{i}^{2}}{(b_{i}-a_{i})^{2}}\bigg\}\\
&\quad\quad=\prod_{i=1}^{d}\bigg(\sum_{M_{i}\in\mathbb{N}}e^{-\frac{\pi^{2}M_{i}^{2}}{(b_{i}-a_{i})^{2}}}\bigg)\geq\prod_{i=1}^{d}\bigg(\int_{0}^{\infty}e^{-\frac{\pi^{2}u^{2}}{(b_{i}-a_{i})^{2}}}du-1\bigg)=\prod_{i=1}^{d}\bigg(\frac{b_{i}-a_{i}}{\sqrt{4\pi}}-1\bigg).
\end{split}
\end{equation}
For the second summation in the last line of \eqref{app-equality-1} (or the second summation on the right hand side of \eqref{app-equality-2}), we have
\begin{equation}\label{app-inequality-3}
\begin{split}
\sum_{N\in\mathbb{N}^{n}}{\rm exp}\bigg\{-\pi^{2}\sum_{j=1}^{n}\frac{N_{j}^{2}}{(d_{j}-c_{j})^{2}}\bigg\}=\prod_{j=1}^{n}\bigg(\sum_{N_{j}\in\mathbb{N}}e^{-\frac{\pi^{2}N_{j}^{2}}{(d_{j}-c_{j})^{2}}}\bigg)\leq\prod_{j=1}^{n}\int_{0}^{\infty}e^{-\frac{\pi^{2}u^{2}}{(d_{j}-c_{j})^{2}}}du=\prod_{j=1}^{n}\frac{d_{j}-c_{j}}{\sqrt{4\pi}}.
\end{split}
\end{equation}

The result of the lemma then follows from \eqref{app-equality-1}, \eqref{app-equality-2}, \eqref{app-inequality-1}, \eqref{app-inequality-2} and \eqref{app-inequality-3}.
\end{proof}

\subsection*{Acknowledgements} The author would like to thank Professor Wenxian Shen for valuable suggestions during the work of this paper.

\end{document}